\documentclass[12pt,a4paper,reqno]{amsart}

\usepackage{amssymb}
\usepackage{enumerate}
\usepackage{mathtools}
\usepackage[all]{xy}
\newtheorem{theorem}{Theorem}
\newtheorem{proposition}{Proposition}[section]
\newtheorem{lemma}[proposition]{Lemma}
\newtheorem{corollary}[proposition]{Corollary}

\theoremstyle{definition}
\newtheorem{definition}[proposition]{Definition}
\newtheorem{remark}[proposition]{Remark}
\newtheorem{example}[proposition]{Example}

\newcommand\dashmapsto{\mapstochar\dashrightarrow}

\newcommand{\p}{\mathbb{P}}
\newcommand{\Z}{\mathbb{Z}}
\renewcommand{\k}{\mathbf{k}}
\newcommand{\Aut}{\mathrm{Aut}}
\newcommand{\Taut}{\mathrm{TAut}}
\newcommand{\PGL}{\mathrm{PGL}}

\newcommand{\GL}{\mathrm{GL}}
\newcommand{\Sym}{\mathrm{Sym}}
\newcommand{\A}{\mathbb{A}}
\newcommand{\Jac}{\mathrm{Jac}}
\newcommand{\Aff}{\mathrm{Aff}}
\newcommand{\C}{\mathbb{C}}
\newcommand{\kk}{\overline{\k}}

\newcommand{\Pn}{\p^n_\k}
\newcommand{\car}{\mathrm{char}}

\newcommand{\longto}{\longrightarrow}
\renewcommand{\epsilon}{\varepsilon}

\DeclareMathOperator{\Bir}{Bir}

\newcommand{\Pic}{\mathrm{Pic}}

\title[Standard and linear Cremona transformations]{The group of Cremona transformations generated by linear maps and the standard involution}

\author{J\'er\'emy Blanc}
\author{Isac Hed\'en}

\address{J\'er\'emy Blanc\\
 Mathematisches Institut\\
Universit\"at Basel\\
 Spiegelgasse 1\\
  4051 Basel, Switzerland}
\email{Jeremy.Blanc@unibas.ch}

\address{Isac Hed\'en\\
 Mathematisches Institut\\
Universit\"at Basel\\
 Spiegelgasse 1\\
  4051 Basel, Switzerland}
\email{Isac.Heden@unibas.ch}

\thanks{The authors acknowledge support by the Swiss National Science Foundation Grant  "Birational Geometry" PP00P2\_128422 /1.}
\date{\today}

\begin{document}
\begin{abstract}{This article studies the group generated by automorphisms of the projective space of dimension $n$ and by the standard birational involution of degree $n$. Every element of this group only contracts rational hypersurfaces, but in odd dimension, there are simple elements having this property which do not belong to the group. Geometric properties of the elements of the group are given, as well as a description of its intersection with monomial transformations.}\end{abstract}

%\subjclass{14E07}
\maketitle

\section{Introduction}
Let us fix a ground field $\k$. We will assume it to be arbitrary whenever we do not state otherwise. The Cremona group $\mathrm{Cr}_n(\k)=\Bir(\Pn)$ is the group of birational transformations of the space of dimension $n$. This group contains $\Aut(\Pn)=\PGL(n+1,\k)$ and the birational map given by
$$\sigma_n\colon [x_0:\dots:x_n]\dashmapsto [\frac{1}{x_0}:\dots:\frac{1}{x_n}]=[x_1x_2\dots x_n:x_0x_2x_3\dots x_n:\dots:x_0x_1\dots x_{n-1}].$$
For $n=3$, this map is called \emph{tetrahedral transformation} and written $T_{\mathrm{tet}}$ in \cite[page 301,\S14]{Hudson} or standard \emph{cubo-cubic transformation of space} \cite[page 179]{SempRoth} or \emph{$(3,3)$-Transformation} \cite[pages 2071-2072, 2108]{Enc}, as its degree and the degree of its inverse are three (but of course this map is not the only one having these properties).
Nowadays, the usual terminology is to call $\sigma_n$, in any dimension $n$, \emph{the standard Cremona transformation} (see for instance
\cite{Gizatullin}, \cite{GP}, \cite[Page~72]{Marcolli}, \cite{Dolga},�\cite{Deserti}). The map $\sigma_n$ restricts to an automorphism of the standard torus $T\subset \Pn$ and contracts the $n+1$ coordinate hyperplanes, i.e. the complement of the torus. 

\bigskip

The group $\Bir(\p^1_\k)$ is equal to $\Aut(\p^1_\k)$, so we will always assume $n\ge 2$ in the sequel. The classical Noether-Castelnuovo theorem \cite{Cas} (see also \cite[Chapter V, \S5, Theorem 2, p. 100]{Sha}) asserts that $\Bir(\p^2_\k)$ is generated by $\Aut(\p^2_\k)$ and $\sigma_2$ when $\k$ is algebraically closed. This is known to be false when $n\ge 3$ or when $\k$ is not algebraically closed. The reason for these two cases is in fact similar: it follows from the description of $\sigma_n$ that every element of the classical group 
$$G_n(\k)=\langle \sigma_n, \Aut(\Pn)\rangle \subset \Bir(\Pn)$$
(whose elements are called \emph{regular} by A.~Coble \cite[page 359, \S4]{CobleII}  and \emph{punctual} by H.~Hudson and P.~Du Val \cite[page 318, \S29]{Hudson},  \cite{DuVal60} and \cite{DuVal81}, see Section~\ref{Sec:Punctual} for a discussion of the terminology) 
contracts only rational hypersurfaces (recall that a irreducible hypersurface is said to be contracted if its image has codimension~$\ge 2$). Moreover, if $n\ge 3$ or if $\k$ is not algebraically closed, there are elements of $\Bir(\Pn)$ which contract non-rational hypersurfaces  (see Section~\ref{Sec:NotRational}). Hence $G_n(\k)$  is a proper subgroup of $\Bir(\Pn)$ in general. Its elements have been studied in detail in many texts. See in particular \cite{Kantor}, \cite{CobleII}, \cite{Hudson}, \cite{DuVal33}, \cite{DuVal60}, \cite{DuVal81}, \cite{DolOrt}, \cite{Gizatullin}, \cite{Deserti}.

\bigskip

To our knowledge, until now, there has been no other way of showing that elements of $\Bir(\Pn)$ do not belong to $G_n(\k)$ than to look at non-rational hypersurfaces that are contracted. The natural question that arises is then whether the above reason is the only one which prevents elements from being in $G_n(\k)$, i.e.: does $G_n(\k)$ contain every element of $\Bir(\Pn)$ that contracts only rational hypersurfaces?

The answer to this question is positive for $n=2$ and any field $\k$; this follows from an adaptation of the proof of the Noether-Castelnuovo theorem (see Proposition~\ref{Prop:G2Krational} below). However, we show in this text that the answer is negative for any odd integer $n\ge 3$ and any field $\k$. For instance we prove that for $n\geq 2$, the birational monomial map 
$$\xi_n \colon [x_0:\dots:x_n]\dashmapsto [x_0:x_1 :x_2\frac{x_1}{x_0}:x_3:\dots:x_n]$$
does not belong to $G_n(\k)$ if $n$ odd, but it does if $n$ is even and $\car(\k)\not=2$. More generally, we give a way of deciding whether a monomial map belongs to $G_n(\k)$ or not, at least when $\car(\k)\not=2$. Recall that the group of monomial transformations of $\Pn$ is naturally isomorphic to $(\k^*)^n\rtimes\GL(n,\Z)$. The element
$$\left((\alpha_1,\dots,\alpha_n),\left(\begin{array}{llll}
a_{11}& \dots & a_{1n}\\
\vdots & \ddots & \vdots \\
a_{n1} & \dots & a_{nn}\end{array}\right)\right)\in (\k^*)^n\rtimes\GL(n,\Z)$$
corresponds to the birational map
$$ [x_0:\dots:x_n]\dashmapsto [1:\alpha_1 (\frac{x_1}{x_0})^{a_{11}}\cdots (\frac{x_n}{x_0})^{a_{1n}}:\dots:\alpha_n(\frac{x_1}{x_0})^{a_{n1}}\cdots (\frac{x_n}{x_0})^{a_{nn}}].$$
With this natural isomorphism, we obtain the following result.
\begin{theorem}\label{Thm:Monomial}
Let $\k$ be any field and $n\ge 2$.\\
$(1)$ If $n$ is even and $\car(\k)\not=2$, every monomial transformation of $\Pn$ belongs to $G_n(\k)$. \\
$(2)$ If $n$ is odd, there are monomial transformations of $\Pn$ which do not belong to $G_n(\k)$.\\
$(3)$ If $n$ is odd and $\car(\k)\not=2$, then the group of monomial transformations that belong to $G_n(\k)$ is equal to $$(\k^*)^n\rtimes\GL(n,\Z)_{\mathrm{odd}},$$ where $\GL(n,\Z)_{\mathrm{odd}}$ is the subgroup of $\GL(n,\Z)$ consisting of matrices such that each column has an odd number of odd entries $($or an odd sum of entries$)$.
\end{theorem}
\begin{remark} The group $\GL(n,\Z)_{\mathrm{odd}}$ is a maximal subgroup of $\GL(n,\Z)$, and has index $2^n-1$ (Lemma~\ref{Lemm:GLNodd}).\end{remark}
 
The key point in the proof of part (2) of Theorem~\ref{Thm:Monomial} is to observe that in odd dimension, the discrepancy of hypersurfaces that are contracted by elements of $G_n(\k)$ is even (see Section~\ref{Sec:ResolutDiscrepancy}, and Proposition~\ref{Prop:Discrepancy} in particular). This result gives strong geometric properties of elements of $G_n(\k)$, which we describe after giving the following definition.
\begin{definition} 
Let $\varphi\in \Bir(\Pn)$ and let $H,\Gamma\subset \Pn$ be two closed irreducible subsets. Denote by $\pi\colon X\to \Pn$ the blow-up of $\Gamma$ and by $E=\pi^{-1}(\Gamma)$ the exceptional divisor.

We will say that \emph{$\varphi$ sends $H$ onto the exceptional divisor of $\Gamma$} if the restriction of $\pi^{-1}\varphi$ induces a birational map $H\dasharrow E$ $($this implies that $H$ is a hypersurface$)$.
 \end{definition}
 \begin{theorem}\label{Theorem:Contraction}Let $\k$ be any field and $n\ge 2$. Then the following hold.\\ 
 $(1)$ If $n$ is odd and $H\subset \Pn$ is a irreducible hypersurface which is sent by an element $g\in G_n(\k)$ onto the exceptional divisor of an irreducible closed subset $\Gamma\subset \Pn$, then $\Gamma$ has even dimension.\\ 
 $(2)$ If $n> m\ge 0$, $nm$ is even and $\car(\k)\not=2$, then there exists an irreducible closed linear subset $\Gamma\subset \Pn$ of dimension $m$ and an element $g\in G_n(\k)$ that sends a hyperplane onto the exceptional divisor of $\Gamma$. 
 \end{theorem}
Part (1) of Theorem~\ref{Theorem:Contraction} gives the geometric explanation of the fact that $\xi_n\notin G_n(\k)$ if $n$ is odd (see Corollary~\ref{Coro:XinNotInGnIfnodd}). However, it follows from part (2) that there are elements of $G_4(\k)$ which send hypersurfaces onto the exceptional divisors of points, lines or planes. For this reason, it could a priori be possible that $G_4(\k)$ contains all elements of $\Bir(\p^4_\k)$ that contract only rational hypersurfaces. The same question remains open in any even dimension $n\ge 4$. It would also be interesting to know if the group generated by $G_n(\k)$ and $\xi_n$ contains all elements of $\Bir(\Pn)$ that contract only rational hypersurfaces, or to describe   $G_n(\k)\cap \GL(n,\mathbb{Z})$ for $n\ge 3$ and $\car(\k)=2$.
 
 \bigskip
 
 The article is organised as follows.

The result on discrepancies can also be viewed algebraically, using Jacobians. We do this in Section~\ref{Sec:Jacobian}, before giving the geometric description of discrepancies in Section~\ref{Sec:ResolutDiscrepancy}. 
 
 Section~\ref{Sec:ChangingDiscr} explains how one can change the discrepancies in higher dimension, and especially in even dimension. It contains the proofs of Theorems~\ref{Thm:Monomial} and \ref{Theorem:Contraction}.
 
Section~\ref{Sec:Linear} describes when the canonical injections $\Bir(\p^n_\k)\to \Bir(\p^{n+1}_\k)$ send $G_n(\k)$ into $G_{n+1}(\k)$. In Section~\ref{Sec:AutAn}, we show that many automorphisms of $\A^n_\k$ can be obtained as elements of $G_n(\k)$ if $\car(\k)\not=2$, namely all tame automorphisms and the Nagata automorphism. The case of characteristic zero is an easy observation, but the general case demands a little bit more work. 

 Section~\ref{Sec:NotRational} describes the relation between $G_n(\k)$ and rational hypersurfaces which are contracted. Section~\ref{Sec:Punctual}, finally, explains the difference between the classical definitions of punctual maps that appear in the literature.
 
 The authors thank Serge Cantat, Julie D\'eserti, Igor Dolgachev, St\'ephane Lamy, Ivan Pan and Thierry Vust for interesting discussions on the topic and their remarks and corrections on the article. We also express our gratitude to the anonymous referee for his careful reading and very helpful suggestions, remarks and historical references.
\section{The Jacobian}\label{Sec:Jacobian}
\begin{definition}\label{Defi:Jac}
$(a)$ Let $f_0,\dots,f_n\in \k(x_0,\dots,x_n)$ be rational functions. 
We define $$\Jac(f_0,\dots,f_n)=\det\left(\left(\frac{\partial f_i}{\partial x_j}\right)_{i,j=0}^n\right)\in \k(x_0,\dots,x_n).$$

$(b)$ If $\varphi\in \Bir(\Pn)$ is given by
$$[x_0:\dots:x_n]\dashmapsto [f_0(x_0,\dots,x_n):\dots:f_n(x_0,\dots,x_n)]$$
where the $f_i$ are homogeneous polynomials of degree $d$ without common factor, the Jacobian $\Jac(\varphi)$ of $\varphi$ is defined to be $\Jac(f_0,\dots,f_n)$. It is defined up to multiplication with the $(n+1)$-th power of an element of $\k^*$, and has degree $(n+1)(d-1)$ (or can be zero, if $\car(\k)>0$). 
\end{definition}

\begin{remark}
If $\car(\k)=0$, the Jacobian  $\Jac(\varphi)$ of $\varphi\in \Bir(\Pn)$ is a polynomial which determines the hypersurfaces of $\Pn$ where the map $\varphi$ is not locally an isomorphism (this is false in positive characteristic, when the Jacobian is zero).\end{remark}
\begin{lemma}\label{Lemm:Jachom}
Let $\k$ be a field of characteristic zero, let $h\in \k[x_0,\ldots,x_n]_d$ be a homogeneous polynomial of degree $d\in \mathbb{N}$ and let $f_0,\ldots,f_n\in \k(x_0,\ldots,x_n)_e$ be homogeneous rational functions of degree $e\in \mathbb{Z}\setminus\{0\}$. Then $$\Jac(hf_0,\ldots,hf_n)=\left(1+\tfrac de\right)\Jac(f_0,\ldots,f_n)h^{n+1}.$$ 
\end{lemma}
\begin{proof}
By assumption we have $f_i(tx_0,\dots,tx_n)=t^ef_i(x_0,\dots,x_n)$ for $i=0,\dots,n$ and for all $t\in\k^*$. Taking the derivative of both sides with respect to $t$ gives $$\sum_{j=0}^n \frac{\partial f_i}{\partial x_j}(tx_0,\dots,tx_n)x_j=et^{e-1}f_i(x_0,\dots,x_n),$$ and evaluating at $t=1$ we see that the $f_i$ satisfy $f_i=\frac{1}{e}\sum_{j=0}^n x_j\frac{\partial f_i}{\partial x_j}$.\par We first show that the result holds when $h=x_0$. Using linearity of $\det$ with respect to the first column, the fact that $f_i=\frac{1}{e}\sum_{j=0}^n x_j\frac{\partial f_i}{\partial x_j}$, and the fact that $\det$ is alternating, we obtain
\begin{eqnarray*}
\Jac(hf_0,\ldots,hf_n)&=&\det\begin{pmatrix}f_0+x_0\frac{\partial f_0}{\partial x_0} & x_0\frac{\partial f_0}{\partial x_1} & \dots & x_0\frac{\partial f_0}{\partial x_n}\\
\vdots &\vdots&&\vdots\\
f_n+x_0\frac{\partial f_n}{\partial x_0} & x_0\frac{\partial f_n}{\partial x_1} & \dots & x_0\frac{\partial f_n}{\partial x_n}\end{pmatrix}\\
&=&(x_0)^{n+1}\Jac(f_0,\dots,f_n)+(x_0)^n\frac 1e\sum_{j=0}^nx_j\det\begin{pmatrix}\frac{\partial f_0}{\partial x_j}&\frac{\partial f_0}{\partial x_1}&\dots&\frac{\partial f_0}{\partial x_n}\\
\vdots&\vdots&&\vdots\\
\frac{\partial f_n}{\partial x_j}&\frac{\partial f_n}{\partial x_1}&\dots&\frac{\partial f_n}{\partial x_n}\end{pmatrix}\\
&=&h^{n+1}\left(1+\tfrac 1e\right)\Jac(f_0,\dots,f_n),
\end{eqnarray*}
so the result holds for $h=x_0$. Analogously it holds for $h=x_j,\,j=1\dots,n$ and it also holds for $h=\lambda\in \k$. Applying it repeatedly, we obtain the case when $h$ is a monomial.

Since the Jacobian is a derivation in each of its arguments and zero whenever two of the arguments are equal, we may use the product rule repeatedly and obtain
\begin{equation*}\label{jac}
\Jac(hf_0,\ldots,hf_n)=h^{n+1}\Jac(f_0,\ldots,f_n)+h^n\sum_{i=0}^nf_i\cdot \Jac(f_0,\ldots,f_{i-1},h,f_{i+1},\ldots,f_n).
\end{equation*}
Hence, the result is equivalent to the following equality:
$$e\sum_{i=0}^n f_i\cdot\Jac(f_0,\ldots,f_{i-1},h,f_{i+1},\ldots,f_n)=d h\cdot \Jac(f_0,\ldots,f_n),$$
which is linear with respect to $h$, so the general case follows from the case of monomials.
\end{proof}
\begin{remark}
The Jacobian is an ancient tool, much used and studied in classical works. In \cite[Page 261]{Muir}, one can see the identity
$$\Jac(fx_0,\dots,fx_n)=f^n(f+x_0\frac{\partial f}{\partial x_0}+\dots +x_n\frac{\partial f}{\partial x_n})$$
for an arbitrary function $f$ (not necessarily homogeneous) and it is written "the proof is hardly needed". If $f$ is homogeneous of degree $m$, then $x_0\frac{\partial f}{\partial x_0}+\dots +x_n\frac{\partial f}{\partial x_n}=mf$ and the right hand side of the identity becomes $(1+m)f^{n+1}$, which is a special case of Lemma~\ref{Lemm:Jachom} (corresponding to the identity map).

For other relations and properties related to the Jacobian, see also \cite[Page~220]{Muir2} and \cite[Pages~228--229]{Pascal}.

\end{remark}
\begin{corollary}\label{jacsigma_n}If $\car(\k)=0$, then $\Jac(\sigma_n)=n(-1)^n\prod_{i=0}^n (x_i)^{n-1}.$
\end{corollary}
\begin{proof} Since $\sigma_n=[\frac h{x_0}:\dots:\frac h{x_n}]$ with $h=\prod_{i=0}^nx_i$, it follows by Lemma~\ref{Lemm:Jachom} that
$$\Jac(\sigma_n)=\left(1+\tfrac {n+1}{-1}\right)\Jac(x_0^{-1},\dots,x_n^{-1})h^{n+1}
=n(-1)^n\prod_{i=0}^n (x_i)^{n-1}.
$$
\end{proof}
\begin{remark}
One can check by hand that Corollary~\ref{jacsigma_n} also holds if $\car(\k)>0$, even if Lemma~\ref{Lemm:Jachom} does not hold, but this will not be used in the sequel.
\end{remark}
%\begin{lemma}\label{jacsigma_n} The Jacobian of $\sigma_n$ is $n(-1)^n\prod_{i=0}^n (x_i)^{n-1}$
%\end{lemma}
%\begin{proof} By definition of the determinant
%$$
%\Jac(\sigma_n)=\sum_{\tau\in\Sym_{n+1}}(-1)^\tau\prod_{i=0}^n\frac{\partial(\sigma_n)_i}{\partial x_{\tau(i)}},
%$$ and it is not difficult to see that
%$$\prod_{i=0}^n\frac{\partial(\sigma_n)_i}{\partial x_{\tau(i)}}=\begin{cases}\prod_{i=0}^n (x_i)^{n-1}, &\textrm{if }\tau(i)\neq i \textrm{ for all } i\\ 0,&\textrm{if } \tau(i)= i\textrm{ for some }i.\end{cases}
%$$
%It follows that
%\begin{eqnarray*}
%\Jac(\sigma_n)&=&\prod_{i=0}^{n}(x_i)^{n-1}\det\begin{pmatrix}0&1&1&\dots&1\\1&0&1&\dots&1\\1&1&0&\dots&1\\\vdots&\vdots&&\ddots&\\1&1&1&\dots&0\end{pmatrix}
%=\prod_{i=0}^n (x_i)^{n-1}\cdot n\det\begin{pmatrix}1&1&1&\dots&1\\1&0&1&\dots&1\\1&1&0&\dots&1\\\vdots&\vdots&&\ddots&\\1&1&1&\dots&0\end{pmatrix}\\
%&=&\prod_{i=0}^n (x_i)^{n-1}\cdot n\det\begin{pmatrix}1&1&1&\dots&1\\0&-1&0&\dots&0\\0&0&-1&\dots&0\\\vdots&\vdots&&\ddots&\\0&0&0&\dots&-1\end{pmatrix}=n(-1)^n\prod_{i=0}^n (x_i)^{n-1}.
%\end{eqnarray*}
%\end{proof}
\begin{proposition}\label{Prop:Hnk}
Suppose that $\car(\k)=0$ and let  $n,k\ge 2$ be some integers such that $k$ divides $n+1$.

\begin{enumerate}[$(1)$]
\item
The set 
\[H_{n,k}=\{f\in \Bir(\p^n)\mid \Jac(f)=\lambda h^{k}\textrm{ for some }h\in \k[x_0,\dots,x_n], \lambda\in \k^*\}\]
is a subgroup of $\Bir(\p^n)$ that contains $\Aut(\p^n)$.
\item
The group $H_{n,k}$ contains $\sigma_n$ if and only if $k=2$ and $n$ is odd.
\item
For each irreducible hypersurface $S\subset \p^n$ of degree $d$ prime to $n+1$, the group $\Aut(\p^n\setminus S)$  is contained in $H_{n,n+1}$, but $\sigma_n\notin H_{n,n+1}$.
\item
In particular, taking $S$ to be a line in $\p^2$, we have $$\langle\Aut(\p^2),\Aut(\p^2\setminus S)\rangle=\langle\Aut(\p^2),\Aut(\A^2)\rangle \subset H_{2,3} \subsetneq\Bir(\p^2)$$ $($the same for $n\ge 3$ being obvious because of hypersurfaces contracted$)$.
\end{enumerate}
\end{proposition}
\begin{proof} Since the Jacobian of every element of $\Aut(\p^n)$ is an element of $\k^*$, we have $\Aut(\p^n)\subset H_{n,k}$.

Let $f,g\in \Bir(\p^n)$ be two birational maps of degree $d_1$, $d_2$, that we write as
$$f:[x_0:\dots:x_n]\dashmapsto [f_0(x_0,\dots,x_n):\dots:f_n(x_0,\dots,x_n)],$$
$$g:[x_0:\dots:x_n]\dashmapsto [g_0(x_0,\dots,x_n):\dots:g_n(x_0,\dots,x_n)],$$ with polynomials $f_i,g_i\in\k[x_0,\dots,x_n]$ such that both the $f_i$ and the $g_i$ are relatively prime.
Supposing that the composition $fg\in \Bir(\Pn)$ has degree $d_1d_2$, the chain rule states that $$\Jac(fg)=g^{*}(\Jac(f))\cdot \Jac(g),$$
where $g^{*}(\Jac(f))$ is obtained by replacing each $x_i$ with $g_i$ in $\Jac(f)$. If the degree of $fg$ is $d_1d_2-m$, for $m>0$, there is a homogeneous polynomial $h$ of degree $m$ that divides the formal composition of $f$ and $g$. Using Lemma~\ref{Lemm:Jachom} this implies that 
$$\Jac(fg)=\left(\tfrac{d_1d_2-m}{d_1d_2}\right)g^{*}(\Jac(f))\cdot \Jac(g)/h^{n+1}.$$
Since $n+1$ is a multiple of $k$, we see that $f,g\in H_{n,k}$ $\Rightarrow$ $fg\in H_{n,k}$. Moreover, taking $g=f^{-1}$ we obtain that $f\in H_{n,k}\Leftrightarrow f^{-1}\in H_{n,k}$. This concludes the proof of Assertion~$(1)$.

Assertion $(2)$  directly follows from Corollary~\ref{jacsigma_n}.

In order to prove $(3)$, denote by $h\in \k[x_0,\dots,x_n]$ the irreducible homogenous polynomial defining $S$ (which is unique up to multiple by an element of $\k^*$). For each $f\in \Aut(\p^n\setminus S)$, the Jacobian of $f$ only vanishes on $S$, hence $\Jac(f)=\lambda h^m$ for some integer $m\ge 0$ and some $\lambda\in \k^*$. We obtain then $$m\deg(h)=\deg(\Jac(f))=(n+1)(\deg(f)-1)$$ (see Definition~\ref{Defi:Jac}). By assumption, $\deg(h)$ and $n+1$ are coprime, so $n+1$ divides $m$. This implies that $f\in H_{n,n+1}$. 
Assertion $(4)$  corresponds to the special case where $n=2$ and $\deg(h)=1$.
\end{proof}
\begin{corollary}\label{Corollary:JacOdd}
If $n$ is odd and $\car(\k)=0$, the Jacobian of each element of $G_n(\k)$ is equal to $\lambda p^2$, for some $\lambda\in \k$ and some homogeneous polynomial $p\in \k[x_0,\dots,x_n]$.
\end{corollary}
\begin{proof}
Using the notation of Proposition~\ref{Prop:Hnk}, this corresponds to saying that $G_n(\k)$ is contained in the group $H_{n,2}$. This is because $\Aut(\p^n)\subset H_{n,2}$ and $\sigma_n\in H_{n,2}$, as observed in the proposition (or in Corollary~\ref{jacsigma_n}).
\end{proof}
\begin{corollary}\label{Coro:QuadrBirInv}
If $n$ is odd and $\car(\k)=0$, the quadratic birational involution of $\p^n_\k$ given by 
$$[x_0:\dots:x_n]\dashmapsto [\frac{x_1x_2}{x_0}:x_1:x_2:\dots:x_n]=[x_1x_2:x_0x_1:\dots:x_0x_n]$$
does not belong to $G_n$.
\end{corollary}
\begin{proof}
The Jacobian of this map is $-2x_0^{n-1}x_1x_2.$ The result follows then from Corollary~\ref{Corollary:JacOdd}.
\end{proof}

\begin{example}
We will see in Example~\ref{Exam:Gizatullin} that the map
$$\theta_n\colon [x_0:\dots:x_n]\dashmapsto [x_0x_1:(x_0)^2:x_1x_2:\dots:x_1x_n]$$ belongs to $\in G_n(\k)$.
It satisfies $\Jac(\theta_n)=-2(x_0)^2\cdot (x_1)^{n-1}$ and this shows in particular that Corollary~\ref{Corollary:JacOdd} cannot be generalised to even dimension. 
\end{example}

Note that Corollary~\ref{Coro:QuadrBirInv} can be easily generalised to the following:
\begin{corollary}\label{Coro:XYnot}
Let $P_1,P_2\in \k[x_1,\dots,x_n]\setminus \{0\}$ be homogeneous of degree $1$ and $2$ respectively, such that $P_2$ defines a reduced quadric of $\Pn$.
If $n\ge 3$ is odd and $\car(\k)=0$, the quadratic birational transformation of $\p^3_\k$ given by 
\begin{eqnarray*}
[x_0:\dots:x_n]&\dashmapsto&  [\frac{x_0P_1(x_1,\dots,x_n)+P_2(x_1,\dots,x_n)}{x_0}:x_1:x_2:\dots:x_n]\\
&=&[x_0P_1(x_1,\dots,x_n)+P_2(x_1,\dots,x_n):x_0x_1:\dots:x_0x_n]\end{eqnarray*}
does not belong to $G_n(\k)$.\end{corollary}
\begin{proof}
Since $$\Jac\left(\frac{x_0P_1(x_1,\dots,x_n)+P_2(x_1,\dots,x_n)}{x_0},x_1,x_2,\dots,x_n\right)=\frac{\partial }{\partial x_0} \left(\frac{P_2}{x_0}\right)=- \frac{P_2}{(x_0)^2}$$
we have
$$\Jac\left(x_0P_1(x_1,\dots,x_n)+P_2(x_1,\dots,x_n),x_0x_1,x_0x_2,\dots,x_0x_n\right)=-2P_2(x_1,\dots,x_n)\cdot (x_0)^{n-1}$$
by Lemma~\ref{Lemm:Jachom}. The result follows then from Corollary~\ref{Corollary:JacOdd}.
\end{proof}
\begin{remark}
Corollaries~\ref{Coro:QuadrBirInv} and~\ref{Coro:XYnot} also hold in positive characteristic (even in characteristic $2$). This can be observed geometrically, with the tools developed in  Section~\ref{Sec:ResolutDiscrepancy}, but also by computing the \emph{affine Jacobian}: To a birational map 
$$f\colon [x_0:\dots:x_n]\dashmapsto [f_0(x_0,\dots,x_n):\dots:f_n(x_0,\dots,x_n)] $$
we can associate its \emph{affine Jacobian}. This is $$\det\left(\left(\frac{\partial g_i}{\partial x_j}\right)_{i,j=1}^n\right)\in \k(x_1,\dots,x_n),$$
where $g_i(x_1,\dots,x_n)=\frac{f_i(1,x_1,\dots,x_n)}{f_0(1,x_1,\dots,x_n)}$, $i=1,\dots,n$, are the coordinates of the restriction of $f$ to the affine open subset where $x_0=1$. One can check that this affine Jacobian is $(-1)^n(\frac{1}{x_1\dots x_n})^{2}$ for $\sigma_n$, so is a square up to scalar multiple (and is not zero, even if $\car(\k)>0$). For $n$ odd, one can then see that the same holds for linear automorphisms and for elements of $G_n(\k)$, by using chain rule.
\end{remark}

\section{Resolution of the standard involution and result on discrepancies}\label{Sec:ResolutDiscrepancy}
Let us recall the following easy and well known result: the standard quadratic involution $\sigma_2\colon \p^2_\k\dasharrow \p^2_\k$ can be factorised as the blow-up of the three coordinate points $[1:0:0]$, $[0:1:0]$, $[0:0:1]$ followed by the contraction of the three coordinate lines passing through $2$ of these $3$ points.

We can generalise this observation in the following way.

\begin{proposition}\label{resolution}
Let $n\ge 2$, let $I_n=\{0,\dots,n\}$ and for each $d\in\{0,\dots,n-2\}$ let $\Omega_d$ be the set of all subsets of $I_n$ of size $n-d$. For each element $\Delta \in \Omega_d$, we denote by $X_\Delta\subset \p^n_\k$ the linear subset of dimension $d$ defined by $x_i=0$ for each $i\in \Delta$.

We then define inductively a sequence of birational morphisms $\pi_d \colon Y_{d+1}\to Y_d$, $d=0,\dots,n-2$ in the following way:
\begin{enumerate}[$(1)$]
\item $Y_0=\p^n_\k$ and $\pi_0\colon Y_1\to \p^n_\k$ is the blow-up of all coordinate points, i.e. all sets $X_\Delta,$ where $\Delta\in \Omega_0$.
\item
For $d=1,\dots,n-2$, $\pi_d\colon Y_{d+1}\to Y_d$ is the blow-up of the strict transform of all varieties $X_\Delta,$ where $\Delta\in \Omega_d$.
\end{enumerate}
Let $Y=Y_{n-1}$, let $\pi\colon Y\to \p^n_\k$ denote the composition $\pi=\pi_0\circ \dots \circ \pi_{n-2}$, and denote by $E_\Delta\subset Y$ the strict transform of the exceptional divisor contracted on $X_\Delta$, for each $\Delta$ in some $\Omega_d$. The following holds:
\begin{enumerate}[$(1)$]
\item
The lift $\hat{\sigma}_n=\pi^{-1}\sigma_n\pi$ is a biregular automorphism of $Y$.
\item
For each $i\in \{0,\dots,d-2\}$ and each $\Delta\in \Omega_i$, the automorphism $\hat{\sigma}_n$ exchanges $E_\Delta$ with $E_{I_n\setminus \Delta}$.
\end{enumerate}
\end{proposition}
\begin{proof}
$(1)$ Denote by $\Sym_{n+1}\subset \Aut(\p^n_\k)$ the group of permutations of variables. The variety $Y_0=\p^n_\k$ is covered by the $n+1$ open subsets where $x_i\not=0$, $i=0,\dots,n$, each isomorphic to $\A^n_\k$ and each containing exactly one point blown up by $\pi_0$. We can moreover choose the isomorphism to be given by
$$(y_1,\dots,y_n)\mapsto \tau([1:y_1:\dots:y_n])$$
where $\tau \in \Sym_{n+1}$. The choice of $\tau$ is not unique, there are $n!$ permutations for one given chart. The point blown up by $\pi_0$ is the origin of $\A^n_\k$, so the blow-up of $\A^n_\k$ at the origin naturally embeds into $\A^n_\k\times \p^{n-1}$ and has then $n$ affine charts isomorphic to $\A^n_\k$, each one intersecting exactly one of the $n$ lines $X_\Delta$ passing through the point corresponding to the origin. We can then choose the charts so that the map $\pi_0$ corresponds to 
$$(y_1,\dots,y_n)\mapsto \tau([1:y_1:y_1y_2:\dots:y_1y_n]),$$
the exceptional divisor corresponds to $y_1=0$ and the line to $y_2=\dots=y_n=0$. Each of the $(n+1)n$ charts on $Y_1$ corresponds to a choice of a point $X_{\Delta_0}$ and a line $X_{\Delta_1}$ containing the point. Continuing in this way, we obtain exactly $(n+1)!$ charts on $Y$, parametrised by the elements  $\tau\in \Sym_{n+1}$ (or equivalently by the flags $X_{\Delta_0}\subset X_{\Delta_1}\subset \dots \subset X_{\Delta_{n-1}}$), such that the map $\pi$ corresponds in the corresponding chart to 
$$(y_1,\dots,y_n)\mapsto \tau([1:y_1:y_1y_2:\dots:y_1y_2\cdots y_n]).$$
Since $\sigma_n$ commutes with elements in $\Sym_{n+1}$, we have
$$\begin{array}{rcl}
\sigma_n\tau([1:y_1:y_1y_2:\dots:y_1y_n])&=&\tau([1:\frac{1}{y_1}:\frac{1}{y_1y_2}:\dots:\frac{1}{y_1y_2\cdots y_n}])\\
&=&\tau([y_1\dots y_n:y_2\dots y_n:y_3\cdots y_n:\dots:y_{n-1}y_n:y_n:1]),\end{array}$$
so the restriction of $\hat{\sigma}_n$ yields an isomorphism between each chart of $Y$ with another one.

$(2)$ Since $\hat\sigma_n$ is an automorphism of $Y$, it is enough to consider the blow-ups $\pi_{\Delta}\colon Z_\Delta\to \p^n_\k$  and $\pi_{I_n\setminus \Delta}\colon Z_{I_n\setminus\Delta}\to \p^n_\k$ of $X_\Delta$ and $X_{I_n\setminus\Delta}$ and check that the lift $\sigma'=(\pi_{I_n\setminus \Delta})^{-1}\sigma_n\pi_{\Delta}$ of $\sigma_n$ induces a birational map between the exceptional divisors.
By a change of variables, we may assume that $X_\Delta\subset \p^n_\k$ and $X_{I_n\setminus\Delta}\subset \p^n_\k$ are given respectively by $x_0=x_1=\ldots=x_k=0$ and $x_{k+1}=\ldots=x_n=0$, so the blow-ups of these two subsets are given locally by
$$\begin{array}{rcl}
\A^n_\k&\to&\p^n_\k\\
(y_1,\dots,y_n)& \mapsto & [y_1:y_1y_2:\dots:y_1y_{k+1}:y_{k+2}:\dots:y_n:1],\textrm{ and}\\ \\
\A^n_\k&\to&\p^n_\k\\
(y_1,\dots,y_n)& \mapsto & [1:y_1:\dots:y_k:y_{k+1}y_n:\dots:y_{n-1}y_n:y_n].\end{array}$$
In these coordinates, the exceptional divisors are given by $y_1=0$ and $y_n=0$ respectively and the map $\sigma'$ becomes
$$(y_1,\dots,y_n)\dashmapsto \left(\frac{1}{y_2},\frac{1}{y_3},\dots, \frac{1}{y_{k+1}},\frac{1}{y_{k+2}},\dots,\frac{1}{y_n},y_1\right).$$ 
Restricting to $y_1=0$ we get a birational map between the two exceptional divisors.
\end{proof}

With this description, we can show that the discrepancy of a hypersurface contracted by an element of $G_n(\k)$ is even when $n$ is odd. To explain what this means, we first recall the following definition.
\begin{definition}
Let $\psi\colon X\dasharrow Y$ be a birational map between two smooth projective varieties and let $H\subset X$ be an irreducible hypersurface. We can always find a birational morphism $\pi\colon Z\to Y$, such that $Z$ is a smooth projective variety and $\pi^{-1}\psi\colon X\dasharrow Z$ can be restricted to a birational map $H\dasharrow E$, for some irreducible hypersurface $E\subset Z$. We then define the \emph{discrepancy} of $H$ with respect to $\psi$ by the integer $a$ as the order of vanishing of $K_Z-\pi^{*}(K_Y)$ along $E$.
\end{definition}
\begin{remark}The discrepancy only depends on $\psi$ and $H$, and not on $\pi$ \cite[Proposition-Definition 4.4.1, page 179]{Matsuki}. In general, this definition is often used for terminal $\mathbb{Q}$-factorial singularities and the discrepancy can be a rational number, but we will only need the smooth case here. In particular, the discrepancies that we will consider are all integers.\end{remark}

\begin{remark}
If $Y=\p^n_\k$ and $\pi\colon X\to Y$ is the blow-up of an irreducible subvariety $\Gamma\subset \p^n_\k$ of dimension $d$, and $E\subset X$ is the exceptional divisor, then the discrepancy of $E$ with respect to $\pi$ is $n-d-1.$
\end{remark}

\begin{proposition}\label{Prop:Discrepancy}
Let $\k$ be any field and let $n\ge 3$ be odd. If $g\in G_n(\k)$ and $H\subset \Pn$ is an irreducible hypersurface, the discrepancy of $H$ with respect to $g$ is always even.
\end{proposition}
\begin{remark}
In characteristic $0$, one can see that the discrepancy of an irreducible hypersurface $H\subset \Pn$ with respect to $\varphi\in \Bir(\Pn)$ is the exponent of the equation of $H$ in $\Jac(\varphi)$. Hence, Proposition~\ref{Prop:Discrepancy} is the geometric version of Corollary~\ref{Corollary:JacOdd}. 
\end{remark}
\begin{proof}
Since the automorphisms of $\p^n_\k$ do not change the discrepancy, we only need to show that the discrepancy of $H$ with respect to $f$ differs by an even number from the discrepancy of $H$ with respect to $\sigma_n f$ for $f\in G_n(\k)$.\par
With the same notation as in Proposition~\ref{resolution}, let $Z$ be a smooth projective variety with a birational morphism $\rho:Z\longto Y$ such that the restriction of $(\pi\rho)^{-1}f:\Pn\dasharrow Z$ to $H$ is a birational map $H\dasharrow H_Z$ for some irreducible hypersurface $H_Z\subset Z$. This fact implies that the discrepancy of $H$ with respect to $f$ (respectively to $\sigma_n f$) is equal to the discrepancy of $H_Z$ with respect to $\pi\rho$ (respectively to $\sigma_n\pi\rho=\pi\hat{\sigma}_n\rho$).
\[\xymatrix{
&&Z\ar[d]^\rho&&\\
&&Y\ar[r]^{\hat\sigma_n}\ar[dl]_\pi&Y\ar[dr]^\pi&\\
\p^n_\k\ar@{-->}[r]^f&\p^n_\k\ar@{-->}[rrr]^{\sigma_n}&&&\p^n_\k\\}\]
For $d=0,\ldots,n-2$, let $E_d\in \Pic( Y_{d+1})$ be the exceptional divisor of $\pi_d\colon Y_{d+1}\to Y_d$, which is the sum of $\binom{n+1}{n-d}$ irreducible divisors, contracted by $\pi_d$ onto the strict transforms of the toric $d$-dimensional linear varieties of $\p^n_\k$.

 Using the ramification formula, $K_{Y_{d+1}}=\pi_d^*(K_{Y_d})+(n-1-d)E_d$, repeatedly  we get
\begin{eqnarray*}
K_Y&=&(\pi)^*(K_{\Pn})+E_{n-2}+\sum_{j=0}^{n-3}(n-1-j)(\pi_{j+1}\cdots\pi_{n-2})^*(E_{j}).
\end{eqnarray*}
Since the linear spaces blown up by $\pi_0,\dots,\pi_{n-2}$ are in increasing dimension, the strict transform of each $E_j$ on $Y$ is the same as its total transform $(\pi_j\cdots\pi_{n-2})^*(E_{j-1})$; we will denote it by $\hat{E}_j$ (in the notation of Proposition~\ref{resolution}, $\hat{E}_j=\sum_{\Delta\in \Omega_j} E_\Delta$).  With this notation the above formula becomes
\begin{eqnarray*}
K_Y&=&(\pi)^*(K_{\Pn})+\sum_{j=0}^{n-2}(n-1-j)\hat{E}_{j}.
\end{eqnarray*}
We also denote by $\hat{E}_{n-1}\in \Pic(Y)$ the strict transform of the union of the coordinate hyperplanes. Applying Proposition~\ref{resolution}, we obtain that $(\hat{\sigma}_n)^{*}(\hat{E}_j)=\hat{E}_{n-1-j}$ for ${j=0,\dots,n-1}$. Applying $(\hat{\sigma}_n)^{*}$ to the above formula we get
\begin{eqnarray*}
K_Y&=&(\pi\hat{\sigma}_n)^*(K_{\Pn})+\sum_{j=0}^{n-2}(n-1-j)\hat{E}_{n-1-j}.
\end{eqnarray*}
It remains to compare the coefficients of $H_Z$ in $K_Z-(\pi\rho)^{*}(K_{\Pn})$ and $K_Z-(\pi\hat{\sigma}_n\rho)^{*}(K_{\Pn})$ and to see that the difference is even:

\begin{eqnarray*}
(\pi\hat{\sigma}_n\rho)^{*}(K_{\Pn})-(\pi\rho)^{*}(K_{\Pn})&=&\rho^{*}((\pi\hat{\sigma}_n)^{*}(K_{\Pn})-\pi^{*}(K_{\Pn}))\\
&=&\rho^{*}\left((K_Y-\sum_{j=0}^{n-2}(n-1-j)\hat{E}_{n-1-j})-(K_Y-\sum_{j=0}^{n-2}(n-1-j)\hat{E}_{j})\right)\\
&=&\rho^{*}\left(\sum_{j=0}^{n-2}(n-1-j)\hat{E}_{j}-\sum_{j=0}^{n-2}(n-1-j)\hat{E}_{n-1-j}\right)\\
&=&\rho^{*}\left(\sum_{j=0}^{n-2}(n-1-j)\hat{E}_{j}-\sum_{j=1}^{n-1}j\hat{E}_{j}\right)\\
&=&\rho^{*}\left((n-1)\hat{E}_0-(n-1)\hat{E}_{n-1}+\sum_{j=1}^{n-2}(n-1-2j)\hat{E}_{j}\right).
\end{eqnarray*}
All coefficients of the above sum are even, since $n$ is odd.
\end{proof}
\begin{remark}
The proof above also shows that composing with elements of $G_n(\k)$ do not change the parity of the discrepancies, when $n$ is odd. It can then be used to say that two elements $\varphi_1,\varphi_2\in \Bir(\Pn)$ are not equal, up to right and left multiplication by elements of $G_n(\k)$.
\end{remark}
\begin{lemma}\label{Lem:SentOnAnyDimension}
Let $\k$ be any field and let $n> m\ge 2$.
Then the birational map $$\varphi\colon [x_0:\dots:x_n]\dashmapsto [x_0:x_1 :x_2\frac{x_1}{x_0}:\dots:x_m \frac{x_1}{x_0}:x_{m+1}:\dots:x_n]$$
sends the hyperplane $H_1\subset \Pn$ given by $x_1=0$ onto the exceptional divisor of the linear subspace $\Gamma\subset \Pn$ given by $x_1=x_2=\dots=x_m=0$ and of dimension $n-m$.
\end{lemma}
\begin{proof}
The hyperplane $H_1\subset \p^n_\k$ given by $x_1=0$ is contracted by $\varphi$ onto the linear subspace~$\Gamma$. The blow-up $\pi_\Gamma\colon X_\Gamma\to \Pn$ of this subset is the birational morphism given by the projection of
$$\begin{array}{rcl}
X_\Gamma&=& \{([x_0:\dots:x_n],[y_1:\dots:y_m])\in \Pn\times \p^{m-1}_\k \mid x_iy_j=x_jy_i, i,j\in \{1,\dots,m\}\}\end{array}
$$
onto the first factor. The birational map $(\pi_\Gamma)^{-1}\circ \varphi$ is then
$$([x_0:\dots:x_n])\dashmapsto ([x_0:x_1 :x_2\frac{x_1}{x_0}:\dots:x_m \frac{x_1}{x_0}:x_{m+1}:\dots:x_n],[x_0:x_2:\dots:x_m]),$$ 
so its restriction to $H_1$ yields a birational map from $H_1$ to the exceptional divisor $E=(\pi_\Gamma)^{-1}(\Gamma)$.
\end{proof}
\begin{corollary}\label{Coro:XinNotInGnIfnodd}
Let $\k$ be any field and $n\ge 3$ be odd. Then the birational map $$\xi_n \colon [x_0:\dots:x_n]\dashmapsto [x_0:x_1 :x_2\frac{x_1}{x_0}:x_3:\dots:x_n]$$
does not belong to $G_n(\k)$.
\end{corollary}
\begin{proof}
It follows from Lemma~\ref{Lem:SentOnAnyDimension} that $\xi_n$ sends a hyperplane onto the exceptional divisor of the linear subspace given by $x_1=x_2=0$. The linear subspace having codimension~$2$, the discrepancy obtained is equal to $1$. Thus $\xi_n$ does not belong to $G_n(\k)$, by Proposition~\ref{Prop:Discrepancy}.
\end{proof}

\begin{example}\label{Example:NotSamedegree}
Looking at the geometric description of $\sigma_n$, one can easily produce elements  $g\in G_n(\k)$ which are not symmetrical. We choose for instance $\alpha\in \Aut(\Pn)$ given by
$$ \begin{array}{llll}
\alpha\colon& [x_0:x_1:\dots:x_n]&\mapsto &[x_0:x_1+x_0:x_2+x_1+x_0:x_3:\dots:x_n],\\
 \alpha^{-1}\colon& [x_0:x_1:\dots:x_n]&\mapsto &[x_0:x_1-x_0:x_2-x_1:x_3:\dots:x_n].\end{array}$$
The image of the coordinate points by $\alpha$ and $\alpha^{-1}$ have distinct alignement with respect to the hyperplanes contracted by $\sigma_n$. Indeed, denoting by $p_i$ the point where $x_i$ is the only non-zero value, $\alpha(p_0)=[1:1:1:0:\dots:0]$ has three non-zero coordinates, but this is not the case for $\alpha^{-1}(p_i)$, $i=0,\dots,n$. This implies that the base-points of $g=\sigma_n \alpha \sigma_n$ and $g^{-1}=\sigma_n \alpha^{-1} \sigma_n$ have a different nature. Doing the computation, one gets
\[\begin{array}{llll}
g\colon& [x_0:x_1:\dots:x_n]&\mapsto &[x_0:\frac{x_0x_1}{x_0+x_1}:\frac{x_0x_1x_2}{x_0x_1+x_0x_2+x_1x_2}:x_3:\dots:x_n],\\
 g^{-1}\colon& [x_0:x_1:\dots:x_n]&\mapsto &[x_0:\frac{x_0x_1}{x_0-x_1}:\frac{x_1x_2}{x_1-x_2}:x_3:\dots:x_n].\end{array}\]
If $n=2$, $g$ and $g^{-1}$ are two birational maps of degree $3$, with not the same number of proper base-points ($3$ and $4$ respectively). If $n\ge 3$, we get maps of different degree: $4$ and $3$. This contradicts the expectations of Kantor and Coble on the degree of elements of $G_3(\k)$:

For regular transformations of $3$-space, Kantor claims, in \cite[Page 7, Theorem~IX]{Kantor} "Die Transformationen haben stets in beiden R\"aumen gleich hohe Ordnung." Also one can see \cite[page 2108]{Enc}, "wie die ebenen Cremonaschen Transformationen: eine Transformation und ihre inverse haben ein und diesselbe Ordnung". Similarly, one find in \cite[Page 366 (24)]{CobleII} "For a regular Cremona transformation in $S_k$ the direct and inverse transformation have the same order and ..."
\end{example}
\section{Changing the discrepancies in even dimension and monomial maps}\label{Sec:ChangingDiscr}
As we saw in Proposition~\ref{Prop:Discrepancy}, the discrepancy of a hypersurface which is contracted by an element of $G_{n}$ is always even when $n$ is odd, and the main reason for this is that $\hat\sigma_n$ exchanges the divisors associated with blow-ups of linear subspaces of dimension $k$ with divisors associated with blow-ups of linear subspaces of dimension $n-1-k$. Here, we show how to use this in even dimension to get elements of $G_n(\k)$ that contract divisors on subspaces of codimension $2$, with discrepancy $1$.

\begin{example}\label{Exam:Gizatullin}
We fix $n\ge 2$.
\begin{enumerate}
\item
Let $\alpha_1\in \Aut(\Pn)$ be given by
$$ [x_0:x_1:\dots:x_n]\mapsto [x_0:x_0-x_1:x_2:\dots:x_n].$$
Then, $\alpha_1\sigma_n\alpha_1\sigma_n\alpha_1\in G_n(\k)$ is equal to the quadratic involution 
$$\theta_n\colon [x_0:\dots:x_n]\dashmapsto [x_0:\frac{(x_0)^2}{x_1}:x_2:\dots:x_n].$$
\item
Moreover denoting by $\alpha_2\in \Aut(\Pn)$ the element given by
$$ [x_0:x_1:\dots:x_n]\mapsto [x_0:x_1:x_1-x_2:x_3:\dots:x_n],$$
the map $\theta_n\alpha_2\theta_n\in G_n(\k)$ is equal to the quadratic involution
$$ [x_0:\dots:x_n]\dashmapsto [x_0:x_1:\frac{(x_0)^2}{x_1}-x_2:x_3:\dots:x_n].$$
\end{enumerate}
\end{example}
\begin{remark}
The composition which yields the map $\theta_n$ (but also the map $\alpha_2$) was already constructed, at least for $n=3$, by M.~Gizatullin in \cite[Page 115]{Gizatullin} to find elements of $G_3(\k)$ which contract hyperplanes onto curves (see Section~\ref{Sec:Punctual} for more details). The geometric idea here is that $\sigma_n$ contracts a hyperplane onto a point, which is then moved by $\alpha$ onto a general point of a coordinate line. The point is then blown up and replaced with a hypersurface in the blow-up of the line, which is sent by $\sigma_n$ onto a codimension $2$ subset.
\end{remark}
\begin{lemma}\label{Lem:Tau}
If $n\ge 3$ and $\car(\k)\not=2$, the following map belongs to $G_n(\k)$:
$$\tau\colon [x_0:\dots:x_n]\dashmapsto [x_0:x_1+\frac{x_2x_3}{x_0}:x_2:\dots:x_n]$$
\end{lemma}
\begin{remark}
We do not know if $\tau$ belongs to $G_n(\k)$ when $\car(\k)=2$. Calculations tend to indicate that this is not the case, but we do not have a proof.
\end{remark}
\begin{proof}It follows from Example~\ref{Exam:Gizatullin} that 
$$\tau'\colon [x_0:\dots:x_n]\dashmapsto [x_0:x_1+\frac{(x_2)^2}{x_0}:x_2:x_3:\dots:x_n]$$
belongs to $G_n(\k)$. Hence, writing $\alpha=[x_0:\dots:x_n]\mapsto [x_0:x_1:x_2+x_3:x_3:\dots:x_n]$, we get $$(\tau')^{-1}\alpha^{-1}\tau'\alpha\colon [x_0:\dots:x_n]\dashmapsto [x_0:x_1+\frac{x_3(x_3-2x_2)}{x_0}:x_2:x_3:\dots:x_n],$$
which is equal to $\tau$, up to linear automorphisms.
\end{proof}
The main construction of this section is the following example, which we first describe algebraically, before explaining the geometry behind the construction.

\begin{example}\label{ExampleChi}The following construction works for each integer $n\ge 3$ but is more interesting for $n\ge 4$.

We denote by $\theta_n,\tau_1,\tau_2,\alpha\in \Bir(\Pn)$ the following birational involutions
$$\begin{array}{rcl}
\theta_n&=&[x_0:\dots:x_n]\mapsto [x_0:\frac{(x_0)^2}{x_1}:x_2:\dots:x_n],\\
\tau_1&=&[x_0:\dots:x_n]\mapsto [x_0:-x_1+\frac{x_0x_2}{x_3}:x_2:x_3:\dots:x_n],\\
\tau_2&=&[x_0:\dots:x_n]\mapsto [x_0:x_1:-x_2+\frac{x_1x_3}{x_0}:x_3:\dots:x_n],\\
\alpha&=&[x_0:\dots:x_n]\mapsto [x_0:x_1:x_3-x_2:x_3:\dots:x_n].\end{array}$$
It follows from Example~\ref{Exam:Gizatullin} and Lemma~\ref{Lem:Tau} that the four maps belong to $G_n(\k)$, if $\car(\k)\not=2$. Then, 
$$\chi_0=\sigma_n \alpha\sigma_n \tau_2\sigma_n\tau_1\theta_n\in G_n(\k)$$ is given by
$$\chi_0 \colon [x_0:\dots:x_n]\dashmapsto [\frac{1}{x_0}:\frac{x_1x_3}{x_0(x_1x_2-x_0x_3)}:\frac{x_1(x_2)^2+x_0(x_3)^2-x_0x_2x_3}{x_0(x_3)^3}:\frac{1}{x_3}:\dots:\frac{1}{x_n}].$$
It sends the hyperplane $H_0\subset \Pn$ given by $x_0=0$ onto the exceptional divisor of the plane $P\subset \Pn$ given by $x_3=\dots=x_n=0$. Hence the discrepancy is $n-3$.

Let us explain the geometric idea of this construction.

\begin{enumerate}
\item
 The map $\theta_n$ sends the hyperplane $H_0$ onto a codimension $1$ subset of the exceptional divisor of the linear subspace $R\subset \p^n_\k$  given by $x_0=x_1=0$. 
 \item
 The map $\sigma_n$ exchanges the exceptional divisor of $R$ with the exceptional divisor of the line $L\subset \p^n_\k$ given by $x_2=\dots=x_n=0$; these divisors are naturally birational to $R\times L$, so that the projections are restrictions of the blow-ups.
 
 \item
 The image of $H_0$ by $\theta_n$ corresponds then to a divisor of bidegree $(0,1)$ in $R\times L$; the projection to $R$ is surjective but the projection to $L$ has only one point in its image.
\item
The map $\tau_1$ fixes $R$; its action on the divisor $R\times L$ sends the divisor of bidegree $(0,1)$ onto a divisor of bidegree $(1,1)$.
\item
After applying $\sigma_n$, we apply $\tau_2$, which fixes $L$. It sends the divisor of bidegree $(1,1)$ onto a divisor of bidegree $(1,0)$ in $R\times L$.
\item
After applying $\sigma_n$ again, the divisor of bidegree $(1,0)$ corresponds to a general hypersurface of $R$ ($x_3=x_2$), which we move to a special one ($x_2=0$) by $\alpha$.
\item The last application of $\sigma_n$ allows to go from $x_0=x_1=x_2=0$ to the plane given by $x_3=\dots=x_n=0$.
\end{enumerate}
\end{example}
\begin{remark}
In the above construction, we could replace the maps $\tau_i$ with maps fixing $L$ and $R$ respectively, and acting on the exceptional divisor in the same way. If $\car(\k)=2$, it does not seem to be possible to obtain such elements in $G_n(\k)$.
\end{remark}\bigskip

Starting with the map  $\chi_0\in G_n(\k)$ of
Example~\ref{ExampleChi}, the following construction provides monomial elements of $G_n(\k)$.

\begin{example}\label{Examplechi2}
We use the map $\chi_0\in G_n(\k)$ of 
Example~\ref{ExampleChi} and the following two elements of $G_n(\k)$.
$$\begin{array}{rcl}
\tau&=&[x_0:\dots:x_n]\dashmapsto [x_0:x_1+\frac{x_0x_3}{x_2}:x_2:x_3:\dots:x_n],\\
\alpha&=&[x_0:\dots:x_n]\dashmapsto [x_0:x_1:x_3-x_2:x_3:\dots:x_n].\end{array}$$
Then $\chi_1=\sigma_n\tau \alpha \chi_0\tau\in G_n(\k)$ is the following monomial map

$$\chi_1 \colon [x_0:\dots:x_n]\dashmapsto [x_0:\frac{x_0x_2}{x_3}:-\frac{x_0(x_3)^3}{x_1(x_2)^2}:x_3:\dots:x_n].$$
\end{example}

Then we can obtain the following.
\begin{proposition}\label{Prop:ContainsMonomials}
Assume that $n\ge 3$ and $\car(\k)\not=2$. Then, the following birational maps are elements of $G_n(\k)$:
$$\begin{array}{rcl}
\mu\colon [x_0:\dots:x_n]&\dashmapsto &[x_0:x_1 :x_2(\frac{x_1}{x_0})^2:x_3:\dots:x_n],\\
\nu \colon [x_0:\dots:x_n]&\dashmapsto &[x_0:x_1 :x_2\frac{x_1}{x_0}:x_3\frac{x_1}{x_0}:x_4:\dots:x_n].
\end{array}$$
Moreover, the following birational map belongs to $G_n(\k)$ if and only if $n$ is even:
$$\xi_n \colon [x_0:\dots:x_n]\dashmapsto [x_0:x_1 :x_2\frac{x_1}{x_0}:x_3:\dots:x_n].$$
\end{proposition}
\begin{proof}
$a)$ It follows from Example~\ref{Exam:Gizatullin} that the birational map
$$\theta_n\colon [x_0:\dots:x_n]\dashmapsto [x_0:\frac{(x_0)^2}{x_1}:x_2:\dots:x_n]=[x_1:x_0:x_2\frac{x_1}{x_0}:\dots:x_n \frac{x_1}{x_0}]$$
belongs to $G_n(\k)$. Hence, the maps
 $$\begin{array}{rrcl}
 \varphi_1\colon& [x_0:\dots:x_n]&\dashmapsto &[x_0:x_1:\frac{(x_1)^2}{x_2}:x_3:\dots:x_n],\\
\vphantom{\Big)} \varphi_2\colon& [x_0:\dots:x_n]&\dashmapsto& [\frac{(x_1)^2}{x_0}:x_1:\dots:x_n],\\
\vphantom{\Big)} \varphi_3\colon& [x_0:\dots:x_n]&\dashmapsto& [x_0:x_1:\frac{(x_0)^2}{x_2}:x_3:\dots:x_n],\end{array}$$
all belong to $G_n(\k)$, which implies that $\mu=\varphi_2\varphi_3\varphi_2\varphi_1\in G_n(\k)$. 

$b)$ It follows from Example~\ref{Examplechi2} that the birational map
$$\chi \colon [x_0:\dots:x_n]\dashmapsto [x_0:\frac{x_0x_2}{x_3}:\frac{x_0(x_3)^3}{x_1(x_2)^2}:x_3:\dots:x_n]$$
belongs to $G_n(\k)$, and as before the map
$$\varphi_4 \colon [x_0:\dots:x_n]\dashmapsto [x_0:\frac{(x_3)^2}{x_1}:x_2:x_3:\dots:x_n]$$
also belongs to $G_n(\k)$. Hence, the map
$$\mu\chi\varphi_4\colon [x_0:\dots:x_n]\dashmapsto [x_0:x_2\frac{x_0}{x_3}:x_1\frac{x_0}{x_3}:x_3:\dots:x_n]$$
belongs to $G_n(\k)$. This implies that $\nu\in G_n(\k)$.

$c)$ Multiplying $\nu$ with conjugates by permutations, we obtain that
$$\psi_k\colon [x_0:\dots:x_n]\dashmapsto [x_0:x_1 :x_2\frac{x_1}{x_0}:x_3\frac{x_1}{x_0}:x_4\frac{x_1}{x_0}:\dots:x_{2k-1}\frac{x_1}{x_0}:x_{2k}:\dots:x_n]$$
belongs to $G_n(\k)$ for each $k$ with $3\le 2k-1\le n$. 

If $n=2k-1$, the map $\xi_n$ does not belong to $G_n(\k)$, by Corollary~\ref{Coro:XinNotInGnIfnodd}.

If $n=2k$, the map $\varphi_5(\psi_k)^{-1}\in G_n(\k)$  is equal to
$$[x_0:\dots:x_n]\dashmapsto [x_0:x_1 :x_2:\dots:x_{n-1}:x_n\frac{x_1}{x_0}]$$
where
$$
\varphi_5 \colon [x_0:\dots:x_n]\dashmapsto [x_0:x_1 :x_2\frac{x_1}{x_0}:x_3\frac{x_1}{x_0}:\dots:x_n\frac{x_1}{x_0}].$$
Since $\varphi_5$ is equal to $\theta_n$ up to automorphisms, the map $\xi_n$ belongs to $G_n(\k)$.
\end{proof}

As we already explained in the introduction, the subgroup of $\Bir(\Pn)$ that consists of monomial transformations with coefficient one is isomorphic to $\GL(n,\mathbb{Z})$: each matrix
$$\left(\begin{array}{llll}
a_{11}& \dots & a_{1n}\\
\vdots & \ddots & \vdots \\
a_{n1} & \dots & a_{nn}\end{array}\right)$$
corresponds to the birational map of $\A^n_\k$ which is given by 
$$(x_1,\dots,x_n)\dasharrow (x_1^{a_{11}}\cdots x_n^{a_{1n}},\dots,x_1^{a_{n1}}\cdots x_n^{a_{nn}})$$
and which extends to the birational map of $\Pn$ given by
$$[x_0:\dots:x_n]\dasharrow [1:(\frac{x_1}{x_0})^{a_{11}}\cdots (\frac{x_n}{x_0})^{a_{1n}}:\dots:(\frac{x_1}{x_0})^{a_{n1}}\cdots (\frac{x_n}{x_0})^{a_{nn}}].$$ 
\begin{lemma}\label{Lemm:GLNodd}
Let $n\ge 1$, and denote by $\GL(n,\Z)_{\mathrm{odd}}\subset \GL(n,\mathbb{Z})$ the subgroup of matrices such that each column has an odd number of odd entries. Then the following hold.

\begin{enumerate}[$(1)$]
\item
The group $\GL(n,\Z)_{\mathrm{odd}}$ is a maximal subgroup of $\GL(n,\mathbb{Z})$, and has index $2^{n}-1$.
\item
If $n\ge 3$, then $\GL(n,\Z)_{\mathrm{odd}}$ is generated by the group of matrix permutations and  the following matrices:
$$M_\theta=\left(\begin{array}{llllll}
-1& \\
 & 1 &  \\
&  & \ddots & \\
&&  & 1 & \end{array}\right), \ M_\mu=\left(\begin{array}{llllll}
1&  &  \\
2 & 1 &  \\
&  & \ddots & \\
&&  & 1 & \end{array}\right),\ M_\nu=\left(\begin{array}{llllll}
1&  &  \\
1 & 1 &  \\
1 &  & 1\\
&&& \ddots & \\
&&  & &1  \end{array}\right).$$
\end{enumerate}
\end{lemma}
\begin{proof}
$(1)$ There is a canonical surjective group homomorphism $ \GL(n,\mathbb{Z})\to \GL(n,\mathbb{F}_2)$, which induces an action $\GL(n,\mathbb{Z})$ on the vector space $V=(\mathbb{F}_2)^n$ by right multiplication. As two distinct elements of $V\setminus \{0\}$ are linearly independent, the action on $V\setminus \{0\}$ is doubly transitive.
Moreover, it follows from the definition of $\GL(n,\Z)_{\mathrm{odd}}$ that it is the isotropy group of the point $(1,\dots,1)$, so $\GL(n,\Z)_{\mathrm{odd}}$ has index $2^n-1$ in $\mathrm{GL}(n,\mathbb{Z})$, and acts transitively on $V\setminus \{(0,\dots,0),(1,\dots,1)\}$. It remains to see that this  implies that $\GL(n,\Z)_{\mathrm{odd}}$ is maximal in $\GL(n,\Z)$. An elementary way is to observe that for each $M,N\in \GL(n,\Z)\setminus \GL(n,\Z)_{\mathrm{odd}}$, there exist $A,B\in \GL(n,\Z)_{\mathrm{odd}}$ such that $AM=NB$: choose $B$ that sends $(1,\dots,1)N$ onto $( 1,\dots,1)M$ and observe that $A=NBM^{-1}$ fixes $(1,\dots,1)$. This can also be explained by more general classical results: groups doubly transitive are primitive \cite[Page 149, Theorem 1.9]{Huppert} and stabilisers in primitive groups are maximal \cite[Page 147, Theorem 1.4]{Huppert}.

$(2)$ We assume $n\ge 3$ and denote by $H\subset \GL(n,\Z)_{\mathrm{odd}}$ the group generated by matrix permutations and by $M_\theta,M_\mu,M_\nu$. Our aim is to prove that $\GL(n,\Z)_{\mathrm{odd}}\subset H$.

 Denote by $H'$ the group of elements of $\GL(n,\Z)_{\mathrm{odd}}$ having ${}^t (0,\dots,0,1)$ as the last column. By sending an element of $H'$ onto the submatrix consisting of the first $n-1$ lines and columns, we obtain a group homomorphism $\rho\colon H'\to  \GL(n-1,\mathbb{Z})$. Note that $\rho$ is surjective: we can complete any element of $\GL(n-1,\mathbb{Z})$ by adding entries on the last line so that the corresponding element belongs to $H'\subset \GL(n,\Z)_{\mathrm{odd}}$. The kernel of $\rho$ is then isomorphic to $\mathbb{Z}^{n-1}$,  generated by the matrices
$$\left(\begin{array}{llllll}
1&  & & &  \\
 & 1 &  \\
&  & \ddots & \\
&&  & 1 &\\
2&&&&1 \end{array}\right),\left(\begin{array}{llllll}
1&  & &   \\
 & 1 &  &&\\
&  & \ddots & \\
&&  & 1 &\\
&2&&&1 \end{array}\right),\dots, \left(\begin{array}{llllll}
1&  & &   \\
 & 1 &  &&\\
&  & \ddots & \\
&&  & 1 &\\
&&&2&1 \end{array}\right),$$
and is thus contained in $H$. Moreover, the matrices
$$M_\theta=\left(\begin{array}{llllll}
-1& \\
 & 1 &  \\
&  & \ddots & \\
&&  & 1 & \end{array}\right), \ M'_\nu=\left(\begin{array}{llllll}
1&  &  &\\
1 & 1 &  \\
&  & \ddots & \\
1&&  & 1 & \end{array}\right)$$
belong to $H'\cap H$ and are sent by $\rho$ onto elements that generate $\GL(n,\mathbb{Z})$, together with the image of the permutation matrices of $H'$. It follows that $H'\subset H$.

It remains to take an arbitrary element $g\in \GL(n,\Z)_{\mathrm{odd}}$ and to see that some element of $Hg$ has last column equal to ${}^t (0,\dots,0,1)$; this will imply that $g\in H$. To do this, we replace $g$ with $hg$, where $h\in H$, in the following way.
\begin{enumerate}
\item[Step $1$:]
We multiply $g$ on the left with multiples of permutations and $M_\theta$, so that the last column becomes ${}^t (a_1,\dots,a_n)$ with $ 0\le a_1\le \dots \le a_n$.
\item[Step $2$:]
We multiply on the left with the conjugate of $M_\nu^{-1}$ by a permutation matrix, replace ${}^t (a_1,\dots,a_n)$ with ${}^t (a_1,\dots,a_{n-3},a_{n-2}-a_{n-1},a_{n-1},a_n-a_{n-1})$, and then go back to Step $1$.
\end{enumerate}
It remains to observe that this algorithm always ends with a last column equal to ${}^t (0,\dots,0,1)$. Note that Step $2$ decreases the value of $\sum_i \lvert a_i \rvert$ by $2a_{n-2}$. If $a_{n-2}=0$, then $a_{n-1}\not=a_n$ because the sum is odd, and applying Step $2$ decreases the value of $\max_{i}\lvert a_i \rvert$ (except in the case where $a_{n-1}=0$, which implies that $a_n=1$). Hence, we always decrease the value of the pair $(\sum_i \lvert a_i \rvert, \max_i \lvert a_i \rvert)$ (in a lexicographic order), until we reach ${}^t (0,\dots,0,1)$.
\end{proof}
We can now give the proofs of Theorems~\ref{Thm:Monomial} and \ref{Theorem:Contraction}, which follow from the above results.
\begin{proof}[Proof of Theorem~$\ref{Thm:Monomial}$]
For any $n$ and any $\k$, the group $G_n(\k)$ contains the diagonal automorphisms of $\Pn$, so the question reduces to studying the intersection of $G_n(\k)$ with the group $\GL(n,\Z)$ of monomial transformations with coefficient one.

If $n=2$, then $G_2(\k)$ contains 
$\theta_2\colon [x_0:x_1:x_2]\dashmapsto [x_0:\frac{(x_0)^2}{x_1}:x_2]=[1:\frac{x_0}{x_1}:\frac{x_2}{x_0}]$ (Example~\ref{Exam:Gizatullin}), and thus also the map
$[x_0:x_1:x_2]\dashmapsto [\frac{(x_0)^2}{x_2}:x_1:x_0]=[1:\frac{x_1}{x_0}\cdot \frac{x_2}{x_0}:\frac{x_2}{x_0}]$. The two correspond to 
$$\left(\begin{array}{llllll}
-1&0 \\
0 & 1\end{array}\right), \ \left(\begin{array}{ll}
1& 1 \\
0 &1 \end{array}\right)\in \GL(2,\Z)$$
and generate, together with permutations, the group $\GL(2,\Z)$.

If $n\ge 4$ is even and $\car(\k)\not=2$, the map
$$\xi_n \colon [x_0:\dots:x_n]\dashmapsto [x_0:x_1 :x_2\frac{x_1}{x_0}:x_3:\dots:x_n]$$
belongs to $G_n(\k)$ (Proposition~\ref{Prop:ContainsMonomials}). This map, together with permutations and $\theta_n$, generates $\mathrm{GL}(n,\mathbb{Z})$.

If $n$ is odd, then $G_n(\k)$ does not contain $\xi_n$  (Corollary~\ref{Coro:XinNotInGnIfnodd}), so $\GL(n,\Z)$ is not contained in $G_n(\k)$.

If $n$ is odd and $\car(\k)\not=2$, then $G_n(\k)$ contains the following maps (see Example~\ref{Exam:Gizatullin} and Proposition~\ref{Prop:ContainsMonomials}):
$$\begin{array}{rcl}
\theta_n \colon [x_0:\dots:x_n]&\dashmapsto &[x_0x_1:(x_0)^2:x_1x_2:\dots:x_1x_n],\\
\mu\colon [x_0:\dots:x_n]&\dashmapsto &[x_0:x_1 :x_2(\frac{x_1}{x_0})^2:x_3:\dots:x_n],\\
\nu \colon [x_0:\dots:x_n]&\dashmapsto &[x_0:x_1 :x_2\frac{x_1}{x_0}:x_3\frac{x_1}{x_0}:x_4:\dots:x_n].
\end{array}$$
This shows, together with Lemma~\ref{Lemm:GLNodd}, that $G_n(\k)$ contains the subgroup $\GL(n,\Z)_{\mathrm{odd}}$ of $\GL(n,\mathbb{Z})$ because $M_\theta,M_\mu,M_\nu\in \GL(n,\Z)_{\mathrm{odd}}$ correspond to $\theta_n,\mu,\nu\in G_n(\k)$ respectively. Moreover, this group being maximal in $\GL(n,\mathbb{Z})$, we have $$G_n(\k)\cap \GL(n,\mathbb{Z})=\GL(n,\Z)_{\mathrm{odd}}.$$
\end{proof}

\begin{proof}[Proof of Theorem~$\ref{Theorem:Contraction}$]
$(1)$ If $H\subset \Pn$ is a irreducible hypersurface which is sent by an element $g\in G_n(\k)$ onto the exceptional divisor of an irreducible closed subset $\Gamma\subset \Pn$, the discrepancy of $H$ with respect to $g$ is $n-d-1$, where $d$ is the dimension of $\Gamma$. If $n$ is odd, the discrepancy is even by Proposition~\ref{Prop:Discrepancy}, so $d$ is even.

$(2)$  If $n> m\ge 0$, $nm$ is even and $\car(\k)\not=2$, we want to find an irreducible closed linear subset $\Gamma\subset \Pn$ and an element $g\in G_n(\k)$ that sends a hyperplane onto the exceptional divisor of $\Gamma$. We consider the map
$$\varphi\colon [x_0:\dots:x_n]\dashmapsto [x_0:x_1 :x_2\frac{x_1}{x_0}:\dots:x_{n-m} \frac{x_1}{x_0}:x_{n-m+1}:\dots:x_n]$$
which sends the hyperplane $H_1\subset \Pn$ given by $x_1=0$ onto the exceptional divisor of the linear subspace $\Gamma\subset \Pn$ given by $x_1=x_2=\dots=x_{n-m}=0$ and of dimension $m$.

If $n$ is even, then $\varphi$ belongs to $G_n(\k)$, like all monomomial birational maps of $\Pn$ (Theorem~\ref{Thm:Monomial}). If $n$ is odd, then $m$ is even, so $\varphi$ belongs to $\GL(n,\Z)_{\mathrm{odd}}$, which is contained in $G_n(\k)$ by Theorem~\ref{Thm:Monomial}.
\end{proof}

\begin{example}\label{ExampleDolgachev}
At the end of the introduction of \cite{Gizatullin}, M.~Gizatullin gives an example of quadratic transformation due to I.~Dolgachev, "considered as an analog for
$\p^5$ of the standard quadratic transformation" $\sigma_2$. The transformation is given by
$$\sigma'\colon [x_0:\dots:x_5]\dashmapsto [x_1x_2:x_0x_2:x_0x_1:x_0x_3:x_1x_4:x_2x_5].$$
It sends the three hyperplanes $x_0=0$, $x_1=0$ and $x_2=0$ onto the exceptional divisors of three planes. In affine coordinates, we obtain
$$(x_1,\dots,x_5)\dashmapsto \left(\frac{1}{x_1},\frac{1}{x_2},\frac{x_3}{x_1x_2},\frac{x_4}{x_2},\frac{x_5}{x_1}\right)$$
which corresponds to the matrix
\[\begin{pmatrix*}[r]
-1&0&0&0&0\\
0&-1&0&0&0\\
-1&-1&1&0&0\\
0&-1&0&1&0\\
-1&0&0&0&1\end{pmatrix*}\in \GL(5,\mathbb{Z})_{\mathrm{odd}}\]
so the map belongs to $G_5(\k)$ if $\car(\k)\not=2$, by Theorem~\ref{Thm:Monomial}.
\end{example}
\begin{remark}
Seeing $\p^5$ as $\p(\k[x_0,\dots,x_2]_2)$, where $\k[x_0,\dots,x_n]_d$ denotes the vector space of homogeneous polynomials of degree $d$, we can choose homogenous coordinates $z_{i,j}=x_ix_j$ on it. The map of Example~\ref{ExampleDolgachev} becomes then
$$\begin{array}{cccc}
\sigma'\colon& [z_{00}:z_{11}:z_{22}:z_{12}:z_{02}:z_{01}]&\dashmapsto &[z_{11}z_{22}:z_{00}z_{22}:z_{00}z_{11}:z_{12}z_{00}:z_{02}z_{11}:z_{01}z_{22}]\\
&&=&[\frac{1}{z_{00}}:\frac{1}{z_{11}}:\frac{1}{z_{22}}:\frac{z_{12}}{z_{11}z_{22}}:\frac{z_{02}}{z_{00}z_{22}}:\frac{z_{01}}{z_{00}z_{11}}]\end{array}$$
One obtains a map $\Aut(\p^2)\cup \{\sigma_2\}\to G_5(\k)$ that sends $\sigma_2$ onto $\sigma'$, and a linear automorphism $\alpha$ onto the linear automorphism corresponding to the action of $\alpha$ on $\p(\k[x_0,\dots,x_2]_2)$. This map naturally extends to a group homomorphism $G_2(\k)\to G_5(\k)$, as it was observed by \cite{Gizatullin}.

One can then generalise this construction to any dimension, and send $\sigma_n$ onto the involution $\sigma'\in \Bir(\p(\k[x_0,\dots,x_n]_2)=\Bir(\p^N)$ given by
$$\sigma'\colon [z_{00}:z_{11}:\dots:z_{nn}:\dots:z_{ij}:\dots]\dashmapsto [\frac{1}{z_{00}}:\frac{1}{z_{11}}:\dots:\frac{1}{z_{nn}}:\dots:\frac{z_{ij}}{z_{ii}z_{jj}}:\dots].$$
This element is a monomial transformation of $\p^N$, that belongs to $G_N(\k)$ if $n\not\equiv 3\pmod{4}$.

The natural question that arises is to know if this extends to a group homomorphism $G_n(\k)\to \Bir(\p^N)$ (or $G_n(\k)\to G_N(\k)$ for $n\not\equiv 3\pmod{4}$). One needs for this to know the relations of the group $G_n(\k)=\langle \Aut(\p^n),\sigma_n\rangle$.
\end{remark}
\section{Linear injections $\Bir(\Pn)\to \Bir(\p^{n+1}_\k)$}\label{Sec:Linear}
There is a canonical injection $$\iota_n\colon \Bir(\A^n_\k)\to \Bir(\A^{n+1}_\k),$$ which corresponds to acting on the $n$ first coordinates of $\A^{n+1}_\k$ and to fix the last one: if $\varphi\in \Bir(\A^n_\k)$, then $\iota_n(\varphi)\in \Bir(\A^{n+1}_\k)$ is given by 
$$(\iota_n(\varphi))(x_1,\dots,x_{n+1})=(\varphi(x_1,\dots,x_n),x_{n+1}).$$

Choosing birational maps $\A^n_\k\to \p^{n}_\k$ and $\A^{n+1}_\k\to \p^{n+1}_\k$, we can thus obtain injections $\Bir(\Pn)\to \Bir(\p^{n+1}_\k)$. The simplest one is when we use linear open embeddings $\A^i_\k\to \p^i_\k$, i.e.\ open embeddings such that the pull-back of a general hyperplane is a hyperplane.

\begin{definition}
An injective group homomorphism $\iota\colon \Bir(\Pn)\to \Bir(\p^{n+1}_\k)$ is said to be a \emph{linear embedding} if there exist linear open embeddings $\tau_n\colon \A^n_\k\to \p^{n}_\k$ and $\tau_{n+1}\colon \A^{n+1}_\k\to \p^{n+1}_\k$ such that
$$\iota(\varphi)=\tau_{n+1} \iota_n((\tau_n)^{-1} \varphi \tau_n)(\tau_{n+1})^{-1}$$
for each $\varphi\in \Bir(\Pn)$ $($where $\iota_n\colon \Bir(\A^n_\k)\to \Bir(\A^{n+1}_\k)$ is the canonical embedding as above$)$.
\end{definition}

\begin{proposition}\label{Prop:GnGnp}
Let $\iota\colon \Bir(\Pn)\to \Bir(\p^{n+1})$ be a linear embedding.

$(1)$ If $n$ is even, then $\iota(G_n(\k))\not\subset G_{n+1}(\k)$.

$(2)$ If $n$ is odd and $\car(\k)\not=2$, then $\iota(G_n(\k))\subset G_{n+1}(\k)$.\end{proposition}
\begin{proof}
Let us denote by $\tau_n\colon \A^n_\k\to \p^{n}_\k$ and $\tau_{n+1}\colon \A^{n+1}_\k\to \p^{n+1}_\k$ the linear embdeddings associated to $\iota$.

If we replace $\tau_n$ with another linear embedding $\tau_n' \colon \A^n_\k\to \p^{n}_\k$, we do not change the group $\iota(G_n(\k))$ since $(\tau_n')^{-1}\circ \tau_n$ is an element of $\Aut(\Pn)\subset G_n(\k)$. Similarly, replacing $\tau_n$ with another linear embedding only replaces $\iota(G_n(\k))$ with a conjugate by an element of $\Aut(\p^{n+1}_\k)\subset G_{n+1}(\k)$. We can then assume that $\tau_n,\tau_{n+1}$ are given by
$$\begin{array}{rcl}
\tau_n(x_1,\dots,x_n)&=&[1:x_1:\dots:x_n],\\
\tau_{n+1}(x_1,\dots,x_{n+1})&=&[1:x_1:\dots:x_n:x_{n+1}].\\
\end{array}$$
With this choice, we can see $\iota(\sigma_n)$ locally as 
$$\iota(\sigma_n)\colon [1:x_1:\dots:x_{n+1}]\dashmapsto [1:\frac{1}{x_1}:\dots : \frac{1}{x_n}:x_{n+1}].$$
Hence, $\iota(\sigma_n)\sigma_{n+1}$ is the map 
$$[1:x_1:\dots:x_{n+1}]\dashmapsto [1:x_1:\dots : x_n:\frac{1}{x_{n+1}}],$$ 
which is equal to the map $\theta_{n+1}\in G_{n+1}(\k)$ of Example~\ref{Exam:Gizatullin} up to permutations. 
This shows that $\iota(\sigma_n)\in G_{n+1}(\k)$ for each $n$, so it remains to decide when $\iota(\Aut(\Pn))\subset G_{n+1}(\k)$.

Let us denote by $A_0\subset \Aut(\Pn)$ the subgroup of elements that preserve the hyperplane $H_0\subset \Pn$ given by $x_0=0$. The group $(\tau_n)^{-1} A_0\tau_n\subset \Bir(\A^n_\k)$ is then equal to the group $\mathrm{Aff}_n$ of affine automorphisms of $\A^n_\k$ (generated by $\GL(n,\k)$ and the translations). Since $\iota(\mathrm{Aff}_n)$ is contained in $\Aff_{n+1}$, we obtain that $\iota(A_0)\subset \Aut(\p^{n+1})\subset G_{n+1}(\k)$.

Because $A_0$ acts transitively on the set of hyperplanes of $\Pn$ distinct from $H_0$, it is a maximal subgroup of $\Aut(\Pn)$. Hence, $\iota(G_n(\k))$ is contained in $G_{n+1}(\k)$ if and only if $\iota(\nu)\in G_{n+1}(\k)$ for one element $\nu\in \Aut(\Pn)\setminus A_0$.

We choose $\nu\in \Aut(\Pn)$ to be the involution
$$\nu\colon [x_0:\dots:x_n]\mapsto [x_1:x_0:x_2:\dots:x_n]$$
and obtain 
$$\begin{array}{rrccl}
(\tau_n)^{-1}\nu\tau_n&\colon& (x_1,\dots,x_n)&\dashmapsto& \left(\frac{1}{x_1},\frac{x_2}{x_1},\dots,\frac{x_n}{x_1}\right)\vspace{0.2cm}\\

\iota_n((\tau_n)^{-1}\nu\tau_n)&\colon& (x_1,\dots,x_{n+1})&\dashmapsto& \left(\frac{1}{x_1},\frac{x_2}{x_1},\dots,\frac{x_n}{x_1},x_{n+1}\right)\vspace{0.2cm}\\

\iota(\nu)&\colon& [x_0:x_1:\dots:x_{n+1}]&\dashmapsto& [x_1:x_0:x_2:\dots:x_n:x_{n+1}\frac{x_1}{x_0}].
\end{array}$$

If $n$ is odd and $\car(\k)\not=2$, then $\iota(\nu)$ belongs to $G_{n+1}(\k)$ by Theorem~\ref{Thm:Monomial}.

If $n$ is even, then $\iota(\nu)$ does not belong to $G_{n+1}(\k)$ by Corollary~\ref{Coro:XinNotInGnIfnodd}. 
\end{proof}

\section{Automorphisms of $\A^n_\k$}\label{Sec:AutAn}
Any linear embedding of $\A^n_\k$ into $\p^n_\k$ yields an injective group homomorphism $$\Aut(\A^n_\k)\to \Bir(\p^n_\k).$$ Changing the linear embedding only changes the image by conjugation by an element of $\Aut(\p^n_\k)$. Hence, whether an element of $\Aut(\A^n_\k)$ belongs to $G_n(\k)$ or not does not depend of the linear embedding. In the sequel, we will always talk about extension of automorphisms of $\A^n_\k$ to $\p^n_\k$ via linear embeddings.\\
 
By Jung - van der Kulk's Theorem \cite{Jun42}, \cite{VdK53}, the group $\Aut(\A^2_\k)$ is generated by $\GL(2,\k)$ and by all elementary automorphisms of the form $(x,y)\mapsto (x+p(y),y)$ (where $p\in \k[y]$ is a polynomial). 

This is no longer true for $\Aut(\A^3_\k)$ if $\car(\k)=0$ \cite{ShUm} and still open for $\Aut(\A^n_\k)$ if $n\ge 4$ or if $n=3$ and $\car(\k)>0$. The group $\Taut(\A^n_\k)\subset \Aut(\A^n_\k)$ generated by $\GL(n,\k)$ and elementary automorphisms is then called \emph{Tame group of automorphisms}.

If $\car(\k)=0$ and $n\ge 3$, H. Derksen showed that the group $\Taut(\A^n_\k)$ is generated by affine automorphisms and by $(x_1,\dots,x_n)\mapsto (x_1+(x_2)^2,x_2,\dots,x_n)$ (see \cite[Theorem 5.2.1, Page 95]{vdE}). Hence, we easily obtain that $\Taut(\A^n_\k)$ is contained in $G_n(\k)$ in this case. In $\car(\k)>0$, the result of H. Derksen does not work, as the following  lemma shows.
\begin{lemma}
Let $n\ge 2$ and let $\k$ be a field of characteristic $p>0$.

$(a)$ Let us write $R_p=\k[(x_1)^p,\dots,(x_n)^p]\subset \k[x_1,\dots,x_n]$.
 Then, the following sets of endomorphisms of $\A^n_\k$ are closed under composition:
$$\begin{array}{lllll}
\{&f\colon (x_1,\dots,x_n)\mapsto (f_1,\dots,f_n)&\mid& f_i \in x_1 R_p+\dots + x_n R_p+R_p\mbox{ for }i=1,\dots,n&\},\\
\{&f\colon (x_1,\dots,x_n)\mapsto (f_1,\dots,f_n)&\mid& f_i \in x_1\k+\dots + x_n \k +R_p\mbox{ for }i=1,\dots,n&\}.\end{array}$$

$(b)$ The group generated by affine transformations and all maps of the form $$(x_1,\dots,x_n)\mapsto (x_1+x_2 x_3^{pa_3}\dots x_{n}^{pa_n},x_2,\dots,x_n),\ a_3,\dots,a_n\in \mathbb{N}$$
 is a proper subgroup of $\Taut(\A^n_\k)$.

$(c)$ The group generated by affine transformations and all maps of the form $$(x_1,\dots,x_n)\mapsto (x_1+(x_2)^{mp},x_2,\dots,x_n),\ m\in \mathbb{N}$$
 is a proper subgroup of $\Taut(\A^n_\k)$.
\end{lemma}
\begin{remark}
If $\k=\mathbb{F}_2$, Assertion $(c)$ can also be proven by looking at the bijections induced by the automorphisms on the $\k$-points of $\A^n_\k$, as observed in \cite{MaubachWillems}.
\end{remark}
\begin{proof}
Assertion $(a)$ is given by the fact that if $g(x_1,\dots,x_n)\in R_p$, then $g(f_1,\dots,f_n)\in R_p$, for each $f_1,\dots,f_n\in \k[x_1,\dots,x_n]$.
Assertions $(b)$ and $(c)$ directly follow from $(a)$.
\end{proof}

The question of finding a finite set of elements that generate, together with affine automorphisms, the group $\Taut(\A^n_\k)$ is still open when $\car(\k)>0$ (see \cite{MaubachWillems}).

We will show all tame automorphisms belong to $G_n(\k)$, when $\car(\k)\not=2$. This shows in particular that we only need affine automorphisms and two elements of $\Bir(\A^n_\k)$ (Corollary~\ref{Coro:Tame2Elements}).
\begin{lemma}\label{Lemm:MonomialTame}
Let $n\ge 2$, let $\k$ be such that $\car(\k)\not=2$ and let $a_2,\dots,a_n\in \mathbb{Z}$. 

If $n a_2$ is even, the element of $\Bir(\A^n_\k)$ given by 
$$(x_1,\dots,x_n)\mapsto (x_1+x_2^{a_2}\dots x_n^{a_n},x_2,\dots,x_n)$$
extends to an element of $G_n(\k)$, via any linear embedding $\A^n_\k\to \p^n_\k$.
\end{lemma}
\begin{proof}

If $a_2$ is even, the matrix 
$$\left(\begin{array}{rrrrr}
1 & -a_2 & -a_3& \dots & -a_n\\
 & 1 & a_3 & \dots & a_n\\
 && 1\\
 &&&\ddots \\
 &&&&1\end{array}\right)$$
 belongs to $\GL(n,\Z)_{\mathrm{odd}}$, so the birational map of $\A^n_\k$ given by
 $$\varphi\colon (x_1,\dots,x_n)\dashmapsto (x_1x_2^{-a_2}\dots x_n^{-a_n},x_2x_3^{a_3}\dots x_n^{a_n},x_3,\dots,x_n)$$
 extends to an element of $G_n(\k)$ by Theorem~\ref{Thm:Monomial}. The same holds if $a_2$ is odd but $n$ is even, since in this case all monomial elements belong to $G_n(\k)$, again by Theorem~\ref{Thm:Monomial}.
  We take
  $$l\colon (x_1,\dots,x_n)\mapsto (x_1+1,x_2,\dots,x_n)$$
  and obtain
  $$\begin{array}{rcl}
  l\varphi\colon(x_1,\dots,x_n) &\dashmapsto&(x_1x_2^{-a_2}\dots x_n^{-a_n}+1,x_2x_3^{a_3}\dots x_n^{a_n},x_3,\dots,x_n)\\
  \varphi^{-1} l\varphi\colon(x_1,\dots,x_n) &\dashmapsto&(x_1+x_2^{a_2}\dots x_n^{a_n},x_2,x_3,\dots,x_n),\\
  \end{array}$$
  which concludes the proof.
\end{proof}
\begin{corollary}\label{Cor:EvenTame}
If $n$ is even and $\car(\k)\not=2$, any linear embedding of $\A^n_\k$ to $\p^n_\k$ yields an inclusion $\Taut(\A^n_\k)\subset G_n(\k)$.
\end{corollary}
\begin{proof}
Follows from Lemma~\ref{Lemm:MonomialTame} and from the fact that $\Taut(\A^n_\k)$ is generated by affine automorphisms, which extend to elements of $\Aut(\p^n_\k)\subset G_n(\k)$, and maps of the form $$(x_1,\dots,x_n)\mapsto (x_1+x_2^{a_2}\dots x_n^{a_n},x_2,\dots,x_n).$$
\end{proof}

In the case where $n$ is odd, we can work a little bit more and obtain the same result:
\begin{proposition}\label{Tame:inGn}
If $n\ge 2$ and $\car(\k)\not=2$, any linear embedding of $\A^n_\k$ to $\p^n_\k$ yields an inclusion $\Taut(\A^n_\k)\subset G_n(\k)$.\end{proposition}
\begin{proof}
By Corollary~\ref{Cor:EvenTame}, we can assume that $n$ is odd. The aim is to show that each map of the form
$$P_v=(x_1,\dots,x_n)\mapsto (x_1+x_2^{v_2}\dots x_n^{v_n},x_2,\dots,x_n),$$
extends to an element of $G_n(\k)$, for each $v=(v_2,\dots,v_n)\in \mathbb{N}^{n-1}$. It follows from Lemma~\ref{Lemm:MonomialTame} that this works if one of the $v_i$ is even (apply permutations of coordinates).

We can also get $P_v$ if one of the $v_i$ is equal to $1$. Indeed, for each $(v_3,\dots,v_n)\in \mathbb{N}^{n-2}$, the following maps belong to $G_{n}(\k)$:
$$\begin{array}{llll}
(x_1,\dots,x_n)&\mapsto& (x_1+(x_2+1)^{2}x_3^{v_3}\dots x_n^{v_n},x_2,\dots,x_n),\\
(x_1,\dots,x_n)&\mapsto& (x_1-x_3^{v_3}\dots x_n^{v_n},x_2,\dots,x_n),\\
(x_1,\dots,x_n)&\mapsto& (x_1-(x_2)^2x_3^{v_3}\dots x_n^{v_n},x_2,\dots,x_n),\end{array}$$
so the composition $(x_1,\dots,x_n)\mapsto (x_1+2x_2x_3^{v_3}\dots x_n^{v_n},x_2,\dots,x_n)$ belongs to $G_n(\k)$. As $\car(\k)\not=2$, 
we get  $P_v\in G_n(\k)$.

Then, we observe that 
$$f\colon (x_1,\dots,x_n)\dasharrow (x_1\cdot (x_2)^2,x_2,\dots,x_n)$$
extends to an element of $G_n(\k)$, by Theorem~\ref{Thm:Monomial}. Conjugating 
$$(x_1,\dots,x_n)\mapsto (x_1+x_2^{v_2}\dots x_n^{v_n},x_2,\dots,x_n)$$
by $f$ we obtain
$$(x_1,\dots,x_n)\mapsto (x_1+x_2^{v_2+2}x_3^{v_3}\dots x_n^{v_n},x_2,\dots,x_n).$$ This implies that we can get all elements $P_v$.
\end{proof}
\begin{corollary}\label{Coro:Tame2Elements}
For each $n\ge 2$, the group $\Taut(\A^n_\k)$ is contained in the subgroup of $\Bir(\A^n_\k)$ generated by affine automorphisms and by 
the following two birational involutions
$$\begin{array}{rcl}
(x_1,\dots,x_n)&\dashmapsto & (\frac{1}{x_1},\dots,\frac{1}{x_n})\\
(x_1,\dots,x_n)&\dashmapsto & (\frac{1}{x_1},\frac{x_2}{x_1},\dots,\frac{x_n}{x_1})\end{array}$$

\end{corollary}
\begin{proof}
We fix the linear embedding $\A^n_\k\to \p^n_\k$, $(x_1,\dots,x_n)\mapsto [1:x_1:\dots:x_n]$, that fixes an isomorphism $\Bir(\A^n_\k)\to \Bir(\p^n_\k)$. The image of affine automorphisms corresponds then to the subgroup of $\Aut(\p^n_\k)$ that preserve the hyperplane $x_0$, being the complement of $\A^n_\k$. The group $G_n(\k)$ is then generated by this group, by the automorphism $[x_0:\dots:x_n]\mapsto [x_1:x_0:x_2:\dots:x_n]$ and by $\sigma_n$. Writing the action of these two elements on $\A^n_\k$ gives the result.
\end{proof}
\begin{example}
Let us recall that the famous Nagata automorphism of $\A^3_\k$ is given by
$$N\colon (x,y,z)\mapsto(x+2y(xz-y^2)+z(xz-y^2)^2, y+z(xz-y^2), z).$$
This element was proven to be in $\Aut(\A^3_\k)\setminus \Taut(\A^3_\k)$ if $\car(\k)=0$ in \cite{ShUm}. The case of $\car(\k)>0$ is however still open.
\end{example}
\begin{proposition}Let $\k$ be a field such that $\car(\k)\not=2$.
Taking any linear embedding $\A^3_\k\to \p^3_\k$ sends the Nagata automorphism to an element of $G_3(\k)$.
\end{proposition}
\begin{proof}
Recall the following classical observation: the Nagata automorphism can be seen as an automorphism of the affine plane over the field $\k(z)$, and as such can be decomposed as a composition of elementary automorphisms and affine automorphisms.

Explicitely, we can write
$$\begin{array}{rcccc}
\alpha&\colon& (x,y,z)&\dashmapsto & (x+\frac{y^2}{z},y,z) \\
\beta&\colon& (x,y,z)&\dashmapsto & (x,y+xz^2,z) \\
\end{array}$$
and obtain $N=\alpha^{-1}\beta\alpha$. We fix a linear embedding $(x,y,z)\mapsto [1:x:y:z]$ of $\A^3_\k$ into $\p^3_\k$ and obtain the birational maps 
$$\begin{array}{rllll}
\hat\alpha&\colon& [w:x:y:z]&\dashmapsto & [w:x+\frac{y^2}{z}:y:z] \\
\hat\beta&\colon& [w:x:y:z]&\dashmapsto & [w:x:y+\frac{xz^2}{w^2}:z].
\end{array}$$
The map $\hat\alpha$ belongs to $G_3(\k)$ by Example~\ref{Exam:Gizatullin}. The map $\hat\beta$ also belongs to $G_3(\k)$ as it is the extension of a tame automorphism $\beta\in \Taut(\A^3_\k)$ (Proposition~\ref{Tame:inGn}).
\end{proof}

\begin{remark}
In $\Taut(\A^n_\k)$, one easily gets elements whose inverse does not have the same degree. For example, $$(x_1,\dots,x_n)\mapsto (x_1+(x_2)^2,x_2+(x_3)^2,x_3,\dots,x_n)$$
has degree $2$ but its inverse has degree $4$. This provides other kind of elements  $g\in G_n(\k)$, such that $\deg(g)\not=\deg(g^{-1})$ (see Example~\ref{Example:NotSamedegree}).
\end{remark}

\section{Rational hypersurfaces contracted}\label{Sec:NotRational}
Let us recall why there are elements of $\Bir(\Pn)$, for each $n\ge 3$, that contract non-rational hypersurfaces (\cite[Page 381, \S31]{Hudson}, \cite{Pan99}). This is done for example by taking an irreducible polynomial $q\in \k[x,y]$ of degree $d>1$ that defines a non-rational curve $\Gamma$ of $\A^2$, and by considering the birational map
$$[x_0:\dots:x_n]\dasharrow [x_0 q(\frac{x_1}{x_3},\frac{x_2}{x_3}):x_1:\dots:x_n],$$
 which contracts the hypersurface given by $q(\frac{x_1}{x_3},\frac{x_2}{x_3})\cdot (x_3)^d$, birational to $\Gamma\times \p^{n-2}$. Note that this argument directly implies that the group $\Bir(\Pn)$ is in fact not generated by $\Aut(\Pn)$ and a finite number of elements. However, the question of whether $\Bir(\Pn)$ is generated by maps preserving a fibration is still open.

 If $n=2$ and $\k$ is not algebraically closed, a similar argument works: we take an irreducible polynomial $p\in \k[x]$ of degree $d>1$, having roots in $\kk\setminus \k$, where $\kk$ is the algebraic closure of $\k$. We then consider the birational map of $\p^2_\k$ given by 
$$[x_0:x_1:x_2]\dashmapsto [x_0:x_1:x_2 p(\frac{x_1}{x_0})]$$
which contracts the curve given by $p(\frac{x_1}{x_0})x_0^d=0$, which is irreducible over $\k$ and not rational (but is a union of lines over $\kk$). 

These two observations show that $G_n(\k)\not=\Bir(\Pn)$ if $n\ge 3$ or if $n=2$ and $\k$ is not algebraically closed.

\bigskip

The description of $G_2(\k)$, for any field $\k$ is given in Proposition~\ref{Prop:G2Krational} below, and follows from an adaptation of the classical proofs of Noether-Castelnuovo's theorem, and the following easy observation on the relation between the base-points and curves contracted.
\begin{lemma}\label{Lem:BasePts}
Let $\k$ be any field, and let $\kk$ be its algebraic closure. Let $\varphi\in \Bir(\p^2_\k)$ be a birational map, which is also a birational map of $\p^2_{\kk}$. The following are equivalent.
\begin{enumerate}[$(1)$]
\item
Every $\kk$-base-point of $\varphi$ is defined over $\k$.
\item
Every $\kk$-base-point of $\varphi^{-1}$ is defined over $\k$. 
\item
Every irreducible $\kk$-curve contracted by $\varphi$ is defined over $\k$. 
\item
Every irreducible $\k$-curve contracted by $\varphi$ is rational over $\k$. 
\end{enumerate}
\end{lemma}
\begin{proof}Seeing $\varphi$ as an element of $\Bir(\p^2_{\kk})$, where $\kk$ is the algebraic closure of $\k$, we have a minimal resolution
\[\xymatrix{
&Z\ar[dl]_\pi\ar[dr]^\eta&\\
\p^2_{\kk}\ar@{-->}[rr]^{\varphi}&&\p^2_{\kk}\\}\]
where $\pi$ and $\eta$ are the blow-ups of the $\kk$-base-points of $\varphi$ and $\varphi^{-1}$  respectively. 

$a)$ We suppose that one $\kk$-base-point of $\varphi^{-1}$ is not defined over $\k$. There is thus a $(-1)$-curve  $E\subset Z$, not defined over $\k$, which is contracted by $\eta$ and not by $\pi$. The image $\pi(E)$ is then an irreducible $\kk$-curve in $\p^2_{\kk}$, contracted by $\varphi$. This yields $(3)\Rightarrow (2)$.

$b)$ If one irreducible $\kk$-curve $C$ in $\p^2_{\kk}$ is contracted by $\varphi$ but not defined over $\k$, the closure of $C$ under the $\k$-topology is an irreducible $\k$-curve, not irreducible over $\kk$, and contracted by $\varphi$. This curve being not rational over $\k$, we obtain $(4)\Rightarrow (3)$.

$c)$ Suppose that $(2)$ holds. Observing that the Picard group of $Z$, viewed as a $\kk$-variety, is generated by the pull-back by $\eta$ of the divisor of a line in $\p^2_{\kk}$ and by the exceptional divisors produced by the corresponding blow-ups, we see that each $\kk$-curve on $Z$ is linearly equivalent to a divisor defined over $\k$. Hence, all rational irreducible $\kk$-curves in $Z$ with negative self-intersection are defined over $\k$, since they are rigid in their equivalence classes. This implies that all irreducible $\kk$-curves contracted by $\eta$ and $\pi$ are defined over $\k$, and yields $(1)$, $(3)$ and $(4)$. 

$d)$ The last implication needed is $(1)\Rightarrow (2)$ which is equivalent to $(2)\Rightarrow (1)$, by replacing $\varphi$ with its inverse.
\end{proof}

The classical proofs of the Noether-Castelnuovo's theorem decompose an element of $\Bir(\p^2_{\kk})$ into a product of quadratic and linear elements. One of the simplest can be find in \cite{Alexander}, and works perfectly here if the field is infinite, but not in the case of finite fields, as it uses the choice of "general" points. The same phenomenon occurs for the classical proof of G. Castelnuovo, well explained in  \cite[Chapter 8]{alberich}, together with historical references. We will then take the proofs given by the Minimal Model Program (see \cite[Comment 4, page 622]{Isk96}).

\bigskip

Using the classical Noether inequalities \cite[Lemma 2.4]{Isk96}, one can obtain the decomposition of every element of $\Bir(\p^2_\k)$ into Sarkisov links \cite[Theorem 2.5]{Isk96}. The fact that all our base-points are defined over $\k$ implies that the base-points of the Sarkisov links  are also defined over $\k$. Hence, the links involved are the following \cite[Section 2.2]{Isk96}.

\begin{enumerate}
\item[$(\mathrm{I})$]
A birational map $\p^2_\k\dasharrow \mathbb{F}_1$ given by the blow-up of a $\k$-point of $\p^2_\k$.
\item[$(\mathrm{II})$]
A birational map $\mathbb{F}_m\dasharrow \mathbb{F}_{m\pm 1}$ given by the blow-up of a $\k$-point, followed by the contraction of the strict transform of a fibre.
\item[$(\mathrm{III})$]
A birational morphism $\mathbb{F}_1\to \p^2_\k$ given by the contraction of the $(-1)$-curve onto a $\k$-point of $\p^2_\k$ (inverse of a link of type $\mathrm{I}$).
\item[$(\mathrm{IV})$]
The automorphism $\mathbb{F}_0=\p^1_\k\times \p^1_\k\to \mathbb{F}_0$ that consists of exchanging the two factors.
\end{enumerate}
\begin{remark}\label{RemarkLinkIV}
Usually, there is a fibration associated to each $\mathbb{F}_i$, which is implicit for $i\ge 1$ but important for $\mathbb{F}_0$  as there are two such; this explains why links of type $\mathrm{IV}$ arise. For our purpose, we do not need to really consider the fibration, and observe that the composition of a link with automorphisms is again a link. So we consider automorphisms as composition of zero links, and will then not need links of type $\mathrm{IV}$.
\end{remark}
\begin{lemma}\label{Lem:Links}
Let $\k$ be any field and let $\varphi\in \Bir(\p^2_\k)$ be a birational map, such that all $\kk$-base-points are defined over $\k$. Then, $\varphi$ decomposes into a sequence of elementary links as above, involving only $\p^2_\k$, $\mathbb{F}_1$, $\mathbb{F}_0$.
\end{lemma}
\begin{proof}
As we already explained, we can decompose $\varphi$ into Sarkisov links of type $\mathrm{I},\mathrm{II},\mathrm{III}$ as above. It remains to see that we can avoid $\mathbb{F}_n$ for $n\ge 2$.

To do this, we take two two links of type $\mathrm{II}$ given by
 $$\varphi_1\colon \mathbb{F}_m\dasharrow \mathbb{F}_{m+1}\mbox{ and }\varphi_2\colon \mathbb{F}_{m+1}\dasharrow \mathbb{F}_{m},$$ for some $m\ge 1$ and prove that $\varphi_2\varphi_1$ is either an isomorphism or decomposes into links involving only $\mathbb{F}_i$ for some $i\le m$. If $\varphi_2\varphi_1$ is not an isomorphism, then it has exactly two base-points: the point $p_1$ being the base-points of $\varphi_1$ and the other point $p_2$ corresponding to the base-point of $\varphi_2$, via $\varphi_1$. Note that $p_1$ is a proper point of $\mathbb{F}_m$ and that $p_2$ is either a proper point or infinitely near to $p_1$, and that the two points do not belong to the same fibre (as proper or infinitely near points). The exceptional section $E$ of $\mathbb{F}_m$ is sent by $\varphi_2\varphi_1$ onto itself, so $p_1\in E$ but $p_2$ does not belong to $E$ (as a proper or infinitely near point). If $p_2$ is a proper point of $\mathbb{F}_m$, we factorise $\varphi_1\varphi_2$ into the link associated first to $p_2$ and then to the image of $p_1$, and obtain links $\mathbb{F}_m\dasharrow \mathbb{F}_{m-1}\dasharrow \mathbb{F}_m$. Otherwise, we take a point $p_3\in \mathbb{F}_m$ on a distinct fibre of $p_1$ and not lying on $E$, and denote by $\varphi_3\colon \mathbb{F}_m\dasharrow \mathbb{F}_{m-1}$ the link associated to it, and by $\varphi_4\colon \mathbb{F}_{m-1}\dasharrow \mathbb{F}_{m}$ the link associated to $\varphi_3(p_1)$. Then, $p_1,p_3$ are the two base-points of $\varphi_4\varphi_3$ and $\varphi_2\varphi_1(\varphi_4\varphi_3)^{-1}$ has exactly two base-points, which are now both proper points of $\mathbb{F}_m$. Applying the previous case, we get the result.\end{proof}

\begin{proposition}\label{Prop:G2Krational}
Let $\k$ be some field. The group $G_2(\k)$ is equal to the subgroup of $\Bir(\p^2_\k)$ consisting of elements that contract only rational curves, which is equal to the subgroup of elements of $\Bir(\p^2_\k)$ having all base-points defined over $\k$.
\end{proposition}
\begin{proof}
By Lemma~\ref{Lem:BasePts}, it suffices to show that every element $\varphi\in\Bir(\p^2_\k)$ which has all base-points defined over $\k$ belongs to $G_2(\k)$. If $\varphi$ has degree $1$, then $\varphi\in \Aut(\p^2_\k)=\PGL(3,\k)\subset G_2(\k)$. Otherwise, we apply Lemma~\ref{Lem:Links} and decompose $\varphi$ into links of type $\mathrm{I}$, $\mathrm{II}$, $\mathrm{III}$,  involving only $\p^2_\k,\mathbb{F}_0,\mathbb{F}_1$. We prove then that $\varphi\in G_2(\k)$, proceeding by induction on the number of links, the case of zero links being the case of automorphisms.

Let $\varphi_1\colon \p^2_\k\dasharrow \mathbb{F}_1$ be some link of type $\mathrm{I}$. If $\varphi_1$ is followed by a link $\varphi_2$ of type $\mathrm{III}$, then $\varphi_2\varphi_1$ is an automorphism of $\p^2_\k$, so we can decrease by two the number of links. If $\varphi_1$ is followed by a link $\varphi_2\colon \mathbb{F}_1\dasharrow \mathbb{F}_0$, then $\varphi_3$ is a link $\mathbb{F}_0\dasharrow \mathbb{F}_1$ (we did not use links of type $\mathrm{IV}$, see Remark~\ref{RemarkLinkIV}). It remains to show that the birational map $\psi=(\varphi_1)^{-1}\varphi_3\varphi_2\varphi_1\colon \p^2_\k\dasharrow \p^2_\k$ belongs to $G_2(\k)$. Replacing $\varphi$ with $\varphi\psi^{-1}$ will then decrease by $2$ the number of links needed.

Note that $\psi$ has at most $3$ base-points, and is thus of degree $1$ or $2$. If $\psi$ has degre~$1$, then $\psi\in G_2(\k)$. Otherwise, the three base-points  of $\psi$ are defined over $\k$, and two of these are proper points of $\p^2_\k$: these are the base-point $p_1$ of $\varphi_1$ and $p_2=(\varphi_1)^{-1}(q)$, where $q$ is the base-point of $\varphi_2$, which does not belong to the exceptional divisor of $\mathbb{F}_1$. Moreover, the three base-points of $\psi$ are not collinear, since the linear system of $\psi$ is irreducible and of degree $2$.

If the third base-point $p_3$ of $\psi$ is also a proper point of $\p^2_\k$, there exists $\beta\in \Aut(\p^2_{\k})=\PGL(3,\k)$ that sends the coordinate points $[1:0:0]$, $[0:1:0]$, $[0:0:1]$ onto  $p_1,p_2,p_3$. Then, $\varphi$ and $\sigma_2\beta$ have the same linear system, which implies that $\varphi=\alpha\sigma_2\beta$ for some $\alpha\in \PGL(3,\k)$. This shows that $\varphi\in G_2(\k)$. 

If $p_3$ is not a proper point of $\p^2_\k$, it has to correspond to a tangent to one of the two other points, say $p_1$, and which is not the direction of the line through $p_1,p_2$. We can then as before replace $\varphi$ with $\varphi\beta$, for some $\beta\in\PGL(3,\k)$ and assume that 
 $p_1=[0:0:1]$, $p_2=[0:1:0]$, and that $p_3$ corresponds to the direction of $x_1=0$. Hence, $\varphi=\alpha\theta_2$, where $\alpha\in \PGL(3,\k)$ and $\theta_2$ is as before equal to
$$\theta_2\colon [x_0:x_1:x_2]\dashmapsto [x_0x_1:(x_0)^2:x_2x_1]$$
and belongs to $G_2(\k)$ (see Example~\ref{Exam:Gizatullin}).\end{proof}

\section{Punctual maps}\label{Sec:Punctual}
There are plenty of definitions of a punctual element of $\Bir(\Pn)$ in the literature, which are often falsely considered as equivalent. The most studied case being when $n=3$ and $\k=\mathbb{C}$, let us restrict to it for the moment, and give the most used definitions.

\begin{enumerate}
\item[(0)]
The map $\theta_n$ from Example~\ref{Exam:Gizatullin} can be seen, in affine coordinates, as an affine blow-up $z_1'=\frac{z_1}{z_n}$, $i=1,\dots,n-1$, $z_n'=z_n$. It was studied, at least for $n=3$, by G. Kobb in \cite[page 406]{Kobb}, see also \cite[Pages 2056-2057]{Enc}.
\item
In \cite{Kantor}, S.~Kantor studies birational maps  of $\p^3_\mathbb{C}$ with no curve of the first species, that is transformations without curves whose proper image in the projective space is a surface; this is also the definition of punctual maps given in \cite[Page 116]{Gizatullin}. S. Kantor gives a proof that contains gaps (H. Hudson says that "the proof is admittedly "gewagt"" in \cite[page 318, \S29]{Hudson} using the word "gewagt" that appears in Kantor's work in the footnotes of pages 18 and 20) which says that each such map is generated by $\Aut(\p^3)$ and $\sigma_3$. He also claims that the set of such maps forms a group.
\item
In \cite[page 359, \S4]{CobleII}, A.~Coble defines a \emph{regular} map of $\p^3$ to be a birational map which is a generated by $\Aut(\p^3_\mathbb{C})$ and $\sigma_3$; this terminology is confusing, since regular usually means "defined at every point". A. Coble explains "They are determined essentially by their fundamental points alone and in all important particulars are entirely analogous to the ternary Cremona transformation."  
\item
In \cite[Page 318, \S29]{Hudson},  \cite{DuVal60} and \cite{DuVal81}, H.~Hudson  and P. Du Val define the group of punctual birational map of $\Bir(\p^3_\C)$ as the group generated by $\sigma_3$ and $\Aut(\p^3_\C)$. They claim that the linear system of a map in this group is only defined by points. 
\item
In \cite[page 93]{DolOrt}, I. Dolgachev and D. Ortland define a punctual Cremona transformation to be a map of the form $\pi\tau\eta^{-1}$, where $\eta\colon X\to \p^3_\C$ and $\pi\colon Y\to \p^3_\C$ are blow-ups of points of $\p^3_\C$ and $\tau\colon X\dasharrow Y$ is a pseudo-isomorphism, that is a birational map which does not contract any hypersurface. They also claim that the set of such maps is a group, and compare it to the group generated by $\sigma_n$ and linear automorphisms.  
\end{enumerate}
Recall that an irreducible curve $\Gamma\subset \p^3$ is a fundamental curve of the first kind for $\varphi\in \Bir(\p^3)$ if there is an irreducible surface $S\subset\p^3$ contracted by $\varphi^{-1}$ onto $\Gamma$ (which means that a general point of $S$ is sent to $\Gamma$ by $\varphi^{-1}$ and that a general point of $\Gamma$ is obtained by this way). 

In \cite[Page 116]{Gizatullin}, M. Gizatullin give a simple example which shows that the set of birational maps of $\p^3_\C$ having no fundamental curve of the first kind is not a group, contrary to what S. Kantor claimed to have proven. The map computed in \cite[Page 115]{Gizatullin} is the birational involution
$$\theta\colon [x_0:x_1:x_2:x_3]\dasharrow [(x_1)^2: x_0x_1: x_0x_2: x_0x_3]$$
that we already introduced in Example~\ref{Exam:Gizatullin}.
The map $\theta^{-1}=\theta$ contracts the plane $H_1$ given by $x_1=0$ onto the line $l$ given by $x_0=x_1=0$, or more precisely send $H_1$ on the divisor of a line $\hat{l}$ infinitely near to $l$. Hence, the line $l$ is a fundamental curve of the first kind of $\theta$. The map $\theta$ is then not punctual in all natural senses, except that it belongs to the group generated by $\Aut(\p^3_\C)$ and $\sigma_3$.

This example also shows that the terminology "punctual" does not seem to be appropriate for the group $G_3(\C)$ (and even more for $G_n(\C)$ for $n\ge 4$). Note that H. Hudson was aware of some problems:  in \cite[Page 318, \S39]{Hudson}, after having said that the set of punctual transformations is a group which consists of compositions of maps of the form $\alpha\sigma_3\beta$, $\alpha,\beta\in \Aut(\p^3_\C)$, she describes the linear system of such an element, and assume for this that the base-points of the new map are either points where a surface is contracted by the previous map, or "in general position". The set of such maps obtained could deserve the word "punctual", and are in fact some of the maps described in \cite{DolOrt} (maybe all), but would not be a group.

\bigskip

One consequence of the work made in this text is the following result, which tends to show that the "gewagt" Theorem of S. Kantor is false.
\begin{proposition}
Let $\k$ be any field. There are elements of $\Bir(\p^3_\k)$ which have no fundamental curve of the first species, but which do not belong to $G_3(\k)$.
\end{proposition}
\begin{remark}
The example given here is not really "punctual", since one of the hypersurfaces is contracted onto a line in the exceptional divisor of a point. So the "gewagt" Theorem of Kantor could maybe be true if we avoid such kind of examples (see below).
\end{remark}

\begin{proof}The birational monomial involution
$$\psi \colon [x_0:\dots:x_3]\dashmapsto [x_1:x_0 :x_2\frac{(x_1)^2}{(x_0)^2}:x_3\frac{x_1}{x_0}]$$
does not belong to $G_3(\k)$. Indeed, the map
$$\theta \colon [x_0:x_1:x_2:x_3]\dashmapsto [x_0:x_1 :\frac{(x_3)^2}{x_2}:x_3]$$
belongs to $G_3(\k)$ (Example~\ref{Exam:Gizatullin}), but 
$$\theta \psi\theta\colon [x_0:x_1:x_2:x_3]\dashmapsto [x_1:x_0 :x_2:x_3\frac{x_1}{x_0}]$$
does not belong to $G_3(\k)$  (Corollary~\ref{Coro:XinNotInGnIfnodd}).
The map $\psi$ contracts exactly two hypersurfaces, namely $H_0,H_1\subset \p^3_\k$, given respectively by $x_0=0$ and $x_1=0$, respectively on the points $[0:0:1:0]$ and $[1:0:0:0]$. Hence, $\psi=\psi^{-1}$ does not have any fundamental curve of the first species.
\end{proof}
In fact, part (1) of Theorem~\ref{Theorem:Contraction} shows that even if hypersurfaces are contracted by elements of $G_3(\k)$ onto points, we cannot send a hypersurface onto the exceptional divisor of a curve by some element of $G_3(\k)$, so this group is not "so" far from being punctual. However, the elements of $G_n(\k)$ are very far from being punctual in dimension $n\ge 4$, as one can contract hypersurfaces on the exceptional divisors of planes (part (2) of Theorem~\ref{Theorem:Contraction}).\\

In fact, it seems to us that a "punctual" map should be an element of $\varphi\in \Bir(\Pn)$ admitting a birational morphism $\pi\colon X\to \Pn$, that consists of sequence of blow-ups of points, and such that $(\varphi\pi)^{-1}\colon \Pn\dasharrow X$ should not contract any hypersurface. The strong sense would be to ask that all points lie in $\Pn$, and a weaker sense would allow infinitely near points.

The definition of \cite{DolOrt} goes in this direction, but is in fact even stronger, as it basically ask that the map and also its inverse are punctual in the strong sense (we could say bipunctual). As explained before, the set of maps satisfying any of these definitions is not a group, contrary to what is claimed in many of the texts cited above. The following example also shows that all these definitions are different.
\begin{example}
Let $\k$ be a field with $\car(\k)\not=2$ and let $l\in \Aut(\p^3_\k)$, be the automorphism given by 
$$\begin{array}{lllll}
l&\colon& [x_0:\dots:x_3]&\mapsto& [x_0:x_2+x_3:x_1+x_3:x_1+x_2]\\
l^{-1}&\colon& [x_0:\dots:x_3]&\mapsto& [2x_0:-x_1+x_2+x_3:x_1-x_2+x_3:x_1+x_2-x_3]\end{array}
$$
and observe that $l$ and $l^{-1}$ act in the following way on the coordinate points of $\p^3_\k$:
$$\begin{array}{lll}
l([1:0:0:0])&=&([1:0:0:0])\\
l([0:1:0:0])&=&([0:0:1:1])\\
l([0:0:1:0])&=&([0:1:0:1])\\
l([0:0:0:1])&=&([0:1:1:0])\\
\end{array}\ \ \begin{array}{lll}
l^{-1}([1:0:0:0])&=&([1:0:0:0])\\
l^{-1}([0:1:0:0])&=&([0:-1:1:1])\\
l^{-1}([0:0:1:0])&=&([0:1:-1:1])\\
l^{-1}([0:0:0:1])&=&([0:1:1:-1]).\\
\end{array}$$
Hence, $\alpha=\sigma_3 l \sigma_3$ is punctual in the weak sense defined above, but not in the strong sense, since $l^{-1}$ sends coordinates points onto general points of the hyperplane $x_0=0$. However, $\alpha^{-1}$ is not even punctual in the weak sense, as $l$ sends coordinate points onto points on general points of coordinate lines.

Similarly, taking 
$$\begin{array}{lllll}
l&\colon& [x_0:\dots:x_3]&\mapsto& [x_0-x_2:x_1-x_2:-x_2+x_3:2 x_2-x_3]\\
l^{-1}&\colon& [x_0:\dots:x_3]&\mapsto&[x_0+x_2+x_3:x_1+x_2+x_3:x_2+x_3:2x_2+x_3] \end{array}
$$
the map $\alpha=\sigma_3 l \sigma_3$ is punctual in the strong sense defined above, but $\alpha^{-1}$ is not punctual in the weak one.

\end{example}
The question of showing that every bipunctual map is in fact an element of $G_n(\k)$, asked in \cite{DolOrt} and corresponding to the "gewagt" theorem of Kantor is still open.

\end{document}